\documentclass[12pt,onecolumn]{IEEEtran}
\usepackage{graphicx}
\usepackage{dsfont} %% Enables \mathds
\usepackage{amsthm}
\usepackage{amsmath}
\usepackage{amssymb}
\usepackage{url}
\usepackage{algpseudocode}
\usepackage{color}
\usepackage{enumitem}
\usepackage[utf8]{inputenc} %% Handles special characters in the bibliography
\usepackage{array} %% Adds magic to tabular
\usepackage[format=hang,singlelinecheck=false,caption=false,font=footnotesize]{subfig}
\usepackage{tikz}
\usetikzlibrary{arrows}
\usetikzlibrary{calc}

\usepackage{xcolor}
\definecolor{darkblue}{HTML}{0000AD}
\definecolor{darkgreen}{HTML}{008600}
\definecolor{darkred}{HTML}{8B0000}
\definecolor{darkgray}{HTML}{666666}
\definecolor{_mage}{HTML}{912830}
\definecolor{_cyan}{HTML}{31837a}
\definecolor{_purp}{HTML}{49425c}

\DeclareMathOperator{\id}{id}
\DeclareMathOperator{\derivoper}{d}
\DeclareMathOperator{\intoper}{int}
\DeclareMathOperator{\totvaroperator}{V}

\let\originalleft\left
\let\originalright\right
\DeclareRobustCommand{\left}{\mathopen{}\mathclose\bgroup\originalleft}
\DeclareRobustCommand{\right}{\aftergroup\egroup\originalright}
\newcommand{\R}{\mathbb{R}}
\newcommand{\N}{\mathbb{N}}

\renewcommand{\epsilon}{\varepsilon}
\renewcommand{\phi}{\varphi}

\newcommand{\abs}[1]{\left\lvert #1\right\rvert}
\newcommand{\absb}[1]{\bigl\lvert #1\bigr\rvert}

\newcommand{\norm}[1]{\left\lVert #1\right\rVert}
\newcommand{\normb}[1]{\bigl\lVert #1\bigr\rVert}

\newcommand{\set}[1]{\left\{ #1\right\}}
\newcommand{\setb}[1]{\bigl\{ #1\bigr\}}
\newcommand{\setB}[1]{\Bigl\{ #1\Bigr\}}
\newcommand{\setbb}[1]{\biggl\{ #1\biggr\}}

\newcommand{\paren}[1]{\left( #1\right)}
\newcommand{\parenb}[1]{\bigl( #1\bigr)}
\newcommand{\parenB}[1]{\Bigl( #1\Bigr)}

\newcommand{\sqparen}[1]{\left[ #1\right]}

\newcommand{\floor}[1]{\left\lfloor #1\right\rfloor}

\newcommand{\diff}[1]{\derivoper \!#1}

\newcommand{\rxsymbol}{\mathnormal{\epsilon}}
\newcommand{\stepsymbol}{\mathnormal{h}}
\newcommand{\rx}[1]{{#1^\rxsymbol}}

\newcommand{\rxcustom}[2]{{#1^{#2}}}
\newcommand{\dapprox}[1]{#1^{\rxsymbol,\stepsymbol}}
\newcommand{\dalgo}[1]{#1^\stepsymbol}
\newcommand{\dalgocustom}[2]{#1^{#2}}
\newcommand{\simfunc}[1]{\stackrel{#1}{\sim}}
\newcommand{\equivclass}[2]{\sqparen{#1}_{#2}}
\newcommand{\bdry}[1]{\partial #1}
\newcommand{\setint}[1]{\intoper(#1)}
\newcommand{\indmetric}[1]{\widetilde{d}_{#1}}
\newcommand{\totvar}[1]{\totvaroperator\paren{#1}}
\newcommand{\bigo}[1]{O\paren{#1}}

\newcommand{\into}{\rightarrow}
\newcommand{\goesto}{\rightarrow}

\newcommand{\etal}{et al.}

\newcommand{\see}[1]{(see~#1)}

\newcommand{\sam}[1]{}
\newcommand{\humberto}[1]{}
\newcommand{\ram}[1]{}

\newenvironment{enumparen}{%
  \begin{enumerate}[label=(\arabic*)]%
}{%
  \end{enumerate}%
}

\newenvironment{newitemize}{%
  \begin{itemize}%
}{%
  \end{itemize}%
}

\newtheorem{theorem}{Theorem}

\newtheorem{corollary}[theorem]{Corollary}
\newtheorem{lemma}[theorem]{Lemma}
\newtheorem{assumption}[theorem]{Assumption}
\newtheorem{definition}[theorem]{Definition}

\begin{document}
\title{Metrization and Simulation of Controlled Hybrid Systems}

%% For each author: complete mailing address, telephone number, ELECTRONIC MAIL (EMAIL) ADDRESS and facsimile (fax) number
\author{Samuel~A.~Burden,~\IEEEmembership{Student,~IEEE}, %
  Humberto~Gonzalez,~\IEEEmembership{Member,~IEEE}, %
  Ramanarayan~Vasudevan,~\IEEEmembership{Member,~IEEE}, %
  Ruzena~Bajcsy,~\IEEEmembership{Fellow,~IEEE}, %
  and S.~Shankar~Sastry,~\IEEEmembership{Fellow,~IEEE}%
  \thanks{S. Burden was supported in part by an NSF Graduate Research Fellowship.
    Part of this research was sponsored by the Army Research Laboratory under Cooperative Agreement W911NF-08-2-0004.%
  }
  \thanks{S.~Burden, R.~Bajcsy, and S.~S.~Sastry are with the University of California at Berkeley, EECS Dept. Email: \texttt{\protect\url{sburden,bajcsy,sastry@eecs.berkeley.edu}}.}%
  \thanks{H.~Gonzalez is with the Washington University in St.\ Louis, ESE Dept. Email: \texttt{\protect\url{hgonzale@ese.wustl.edu}}.}%
  \thanks{R.~Vasudevan is with the Massachusetts Institute of Technology, CSAIL\@. Email: \texttt{\protect\url{ramv@csail.mit.edu}}.}%
}

\maketitle

\begin{abstract}
The study of controlled hybrid systems requires practical tools for approximation and comparison of system behaviors.
Existing approaches to these problems impose undue restrictions on the system's continuous and discrete dynamics.
Metrization and simulation of controlled hybrid systems is considered here in a unified framework by constructing a state space metric.
The metric is applied to develop a numerical simulation algorithm that converges uniformly, with a known rate of convergence, to orbitally stable executions of controlled hybrid systems, up to and including Zeno events.
Benchmark hybrid phenomena illustrate the utility of the proposed tools.
\end{abstract}

% sam:
%   1. when referring to the metric, i call it a (theoretical) construction
%   2. when referring to the simulation, i call it a (provably-convergent) algorithm
%   3. when referring to the combination, i call it a (foundational) framework
\section{Introduction}
\label{sec:Introduction}

For continuous--state dynamical systems and finite--state automata there exist rich sets of tools for metrization and simulation.
The interaction of discrete transitions with continuous dynamics introduces subtleties that render the development of similar tools for controlled hybrid systems non-trivial.
Consider the time evolution of a pair of states $(\xi_1,\xi_2)\in\R^2$. Suppose that when either quantity crosses zero a discontinuous change in the time derivatives $(\dot{\xi}_1,\dot{\xi}_2)$ is triggered yielding a discontinuous planar vector field as in Fig.~\ref{fig:dcs}.
A faithful model of the system's full state evolution is \emph{hybrid}, representing both discrete and continuous state transitions, as in Fig.~\ref{fig:chs}.

% Consider the time evolution of two physical quantities $(\xi_1,\xi_2)\in\R^2$ monitored by a digital control system.
% Assume that when either quantity crosses a specified threshold, a logical controller state changes, causing a discontinuous change in the time derivatives $(\dot{\xi}_1,\dot{\xi}_2)$.
% Neglecting the discrete controller component of the system's state evolution yields a discontinuous planar vector field as in Fig.~\ref{fig:dcs}.
% A faithful model of the system's full state evolution is \emph{hybrid}, representing both discrete and continuous state transitions, as in Fig.~\ref{fig:chs}.

The choice of metric for this controlled hybrid system dictates exactly the type of trajectories that can be faithfully simulated.
For example, existing trajectory--space metrics impose at least unit distance between any states that reside in different discrete modes~\cite{Tavernini1987, Gokhman2008, Tavernini2009, SanfeliceTeel2010}.
% include a component that compares only the distance between the discrete labels of pairs of points.
As a result, simulation algorithms based on these metrics cannot provably approximate executions that undergo simultaneous discrete transitions (e.g. $x$) since nearby executions encounter different discrete transition sequences  (e.g. $y_\delta$, $z_\delta$).

\def\quadrant#1#2{
  \begin{scope}[shift={#1},rotate={#2}]
    % \draw[style=dashed,line width=1] (2.0cm,.0cm) arc (0:90:2.0cm);
    \draw[<->,line width=1] (0,2) -- (0,0) -- (2,0);
  \end{scope}
}

\begin{figure}[t]
  \centering
  \mbox{}%
  \hfill%
  \subfloat[digital control system]{\label{fig:dcs}%
    \begin{tikzpicture}
      [scale=0.82,
      dot/.style={circle,fill=darkblue,outer sep=.0mm,inner sep=0mm,minimum size=1.5mm},
      txt/.style={font=\scriptsize},
      rst/.style={style=dashed,line width=1.0,color=darkred},
      trj/.style={line width=1.0,color=darkblue}]

      % grid
      % \draw[step=.5cm,gray,very thin] (-2,-2) grid (2,2);

      \coordinate (o) at (0,0);

      \coordinate (F0) at (1.0,1.0);
      \coordinate (F1) at (1.0,0.8);
      \coordinate (F2) at (0.7,1.0);
      \coordinate (F3) at (1.0,1.0);

      \coordinate (D0o) at (-.0,-.0);
      \coordinate (D1o) at (-.0,+.0);
      \coordinate (D2o) at (+.0,-.0);
      \coordinate (D3o) at (+.0,+.0);

      \node (D0) at ($(D0o) + .85*(-2.,-2.)$) {};
      \node (D1) at ($(D1o) + .85*(-2.,+2.)$) {};
      \node (D2) at ($(D2o) + .85*(+2.,-2.)$) {};
      \node (D3) at ($(D3o) + .85*(+2.,+2.)$) {};

      \node (x0) at ($(D0o) - 1*(F0)$) {};
      \node (x3) at ($(D3o) + 1.3*(F3)$) {};

      \node (x0+) at ($(D0o)$) {};
      \node (x3-) at ($(D3o)$) {};

      \node (y0+) at ($(D0o) + (-1.0,0.)$) {};
      \node (y0-) at ($(y0+) - 1*(F0)$) {};

      \node (y1-) at ($(D1o) + (-1.0,0.)$) {};
      \node (y1+) at ($(y1-) + 1.*(F1)$) {};

      \node (y3-) at ($(y1+) + (D3o) - (D1o)$) {};
      \node (y3+) at ($(y3-) + .5*(F3)$) {};

      \node (z0+) at ($(D0o) + (0.,-1.0)$) {};
      \node (z0-) at ($(z0+) - 1*(F0)$) {};

      \node (z2-) at ($(D2o) + (0.,-1.0)$) {};
      \node (z2+) at ($(z2-) + 1.*(F2)$) {};

      \node (z3-) at ($(z2+) + (D3o) - (D2o)$) {};
      \node (z3+) at ($(z3-) + .4*(F3)$) {};

      % \path[draw,->,line width=2] (-2,0) -- (2,0);
      % \path[draw,->,line width=2] (0,-2) -- (0,2);
      \node[txt] at (1.65cm,3mm) {$\xi_1$};
      \node[txt] at (-3mm,1.65cm) {$\xi_2$};

      \node[txt,color=_purp] at ($(x0) + (.5,0.)$) {$x(0)$};
      \node[txt,color=_purp] at ($(x3) + (.3,.3)$) {$x(T)$};

      \node[txt,color=_cyan] at ($(y0-) + .3*(0.7,-1.)$) {$y_\delta(0)$};
      \node[txt,color=_cyan] at ($(y3+) + (0.1,.3)$) {$y_\delta(T)$};

      \node[txt,color=_mage] at ($(z0-) + (0.5,0.)$) {$z_\delta(0)$};
      \node[txt,color=_mage] at ($(z3+) + (.2,.3)$) {$z_\delta(T)$};

      \node[txt,color=_cyan] at ($.6*(x0) + .4*(y0-) + (0,.2)$) {$\delta$};
      \path[draw,->,rst,color=_cyan] (x0) -- (y0-);

      \node[txt,color=_mage] at ($.6*(x0) + .4*(z0-) + (.2,.0)$) {$\delta$};
      \path[draw,->,rst,color=_mage] (x0) -- (z0-);

      % quadrants
      \quadrant{(D0o)}{180}
      \quadrant{(D1o)}{90}
      \quadrant{(D2o)}{270}
      \quadrant{(D3o)}{0}

      \node[dot,color=_purp] at (D0o) {};
      \node[dot,color=_purp] at (D3o) {};

      \node[dot,color=_purp] at (x0) {};
      \node[dot,color=_purp] at (x3) {};

      \node[dot,color=_cyan] at (y0-) {};
      \node[dot,color=_cyan] at (y0+) {};

      \node[dot,color=_cyan] at (y1-) {};
      \node[dot,color=_cyan] at (y1+) {};

      \node[dot,color=_cyan] at (y3-) {};
      \node[dot,color=_cyan] at (y3+) {};

      \node[dot,color=_mage] at (z0-) {};
      \node[dot,color=_mage] at (z0+) {};

      \node[dot,color=_mage] at (z2-) {};
      \node[dot,color=_mage] at (z2+) {};

      \node[dot,color=_mage] at (z3-) {};
      \node[dot,color=_mage] at (z3+) {};

      \path[draw,trj,color=_purp] (x0) -- (D0o);
      \path[draw,trj,color=_purp] (D3o) -- (x3);
      \path[draw,->,trj,color=_purp] (x0) -- ($ (x0) + .4*(F0) $);
      \path[draw,->,trj,color=_purp] (D3o) -- ($ (D3o) + .4*(F3) $);

      \path[draw,trj,color=_cyan] (y0-) -- (y0+);
      \path[draw,trj,color=_cyan] (y1-) -- (y1+);
      \path[draw,trj,color=_cyan] (y3-) -- (y3+);
      \path[draw,->,trj,color=_cyan] (y0-) -- ($ (y0-) + .4*(F0) $);
      \path[draw,->,trj,color=_cyan] (y1-) -- ($ (y1-) + .4*(F1) $);
      \path[draw,->,trj,color=_cyan] (y3-) -- ($ (y3-) + .4*(F3) $);

      \path[draw,trj,color=_mage] (z0-) -- (z0+);
      \path[draw,trj,color=_mage] (z2-) -- (z2+);
      \path[draw,trj,color=_mage] (z3-) -- (z3+);
      \path[draw,->,trj,color=_mage] (z0-) -- ($ (z0-) + .4*(F0) $);
      \path[draw,->,trj,color=_mage] (z2-) -- ($ (z2-) + .4*(F2) $);
      \path[draw,->,trj,color=_mage] (z3-) -- ($ (z3-) + .3*(F3) $);
    \end{tikzpicture}
  }%
  \hfill%
  \subfloat[controlled hybrid system]{\label{fig:chs}%
    \begin{tikzpicture}
      [scale=0.7,
      dot/.style={circle,fill=darkblue,outer sep=0.0mm,inner sep=0mm,minimum size=1.5mm},
      txt/.style={font=\scriptsize},
      rst/.style={style=dashed,line width=1.0},
      trj/.style={line width=1.0,color=darkblue}]

      % grid
      % \draw[step=.5cm,gray,very thin] (-2,-2) grid (2,2);

      \coordinate (o) at (0,0);

      \coordinate (F0) at (1.0,1.0);
      \coordinate (F1) at (1.0,0.8);
      \coordinate (F2) at (0.7,1.0);
      \coordinate (F3) at (1.0,1.0);

      \coordinate (D0o) at (-.75,-.5);
      \coordinate (D1o) at (-.75,+.5);
      \coordinate (D2o) at (+.75,-.5);
      \coordinate (D3o) at (+.75,+.5);

      \node (D0) at ($(D0o) + .85*(-2.,-2.)$) {};
      \node (D1) at ($(D1o) + .55*(-2.,+2.)$) {};
      \node (D2) at ($(D2o) + .55*(+2.,-2.)$) {};
      \node (D3) at ($(D3o) + .85*(+2.,+2.)$) {};

      \node (x0) at ($(D0o) - 1*(F0)$) {};
      \node (x3) at ($(D3o) + 1.3*(F3)$) {};

      \node (x0+) at ($(D0o)$) {};
      \node (x3-) at ($(D3o)$) {};

      \node (y0+) at ($(D0o) + (-1.0,0.)$) {};
      \node (y0-) at ($(y0+) - 1*(F0)$) {};

      \node (y1-) at ($(D1o) + (-1.0,0.)$) {};
      \node (y1+) at ($(y1-) + 1.*(F1)$) {};

      \node (y3-) at ($(y1+) + (D3o) - (D1o)$) {};
      \node (y3+) at ($(y3-) + .5*(F3)$) {};

      \node (z0+) at ($(D0o) + (0.,-1.0)$) {};
      \node (z0-) at ($(z0+) - 1*(F0)$) {};

      \node (z2-) at ($(D2o) + (0.,-1.0)$) {};
      \node (z2+) at ($(z2-) + 1.*(F2)$) {};

      \node (z3-) at ($(z2+) + (D3o) - (D2o)$) {};
      \node (z3+) at ($(z3-) + .4*(F3)$) {};

      \node[txt] at (D0) {$D_0$};
      \node[txt] at (D1) {$D_1$};
      \node[txt] at (D2) {$D_2$};
      \node[txt] at (D3) {$D_3$};

      % quadrants
      \quadrant{(D0o)}{180}
      \quadrant{(D1o)}{90}
      \quadrant{(D2o)}{270}
      \quadrant{(D3o)}{0}

      \node[dot,color=_purp] at (D0o) {};
      \node[dot,color=_purp] at (D3o) {};

      \node[dot,color=_purp] at (x0) {};
      \node[dot,color=_purp] at (x3) {};

      \node[dot,color=_cyan] at (y0-) {};
      \node[dot,color=_cyan] at (y0+) {};

      \node[dot,color=_cyan] at (y1-) {};
      \node[dot,color=_cyan] at (y1+) {};

      \node[dot,color=_cyan] at (y3-) {};
      \node[dot,color=_cyan] at (y3+) {};

      \node[dot,color=_mage] at (z0-) {};
      \node[dot,color=_mage] at (z0+) {};

      \node[dot,color=_mage] at (z2-) {};
      \node[dot,color=_mage] at (z2+) {};

      \node[dot,color=_mage] at (z3-) {};
      \node[dot,color=_mage] at (z3+) {};

      \path[draw,trj,color=_purp] (x0) -- (D0o);
      \path[draw,trj,color=_purp] (D3o) -- (x3);
      \path[draw,->,trj,color=_purp] (x0) -- ($ (x0) + .4*(F0) $);
      \path[draw,->,trj,color=_purp] (D3o) -- ($ (D3o) + .4*(F3) $);

      \path[draw,trj,color=_cyan] (y0-) -- (y0+);
      \path[draw,trj,color=_cyan] (y1-) -- (y1+);
      \path[draw,trj,color=_cyan] (y3-) -- (y3+);
      \path[draw,->,trj,color=_cyan] (y0-) -- ($ (y0-) + .4*(F0) $);
      \path[draw,->,trj,color=_cyan] (y1-) -- ($ (y1-) + .4*(F1) $);
      \path[draw,->,trj,color=_cyan] (y3-) -- ($ (y3-) + .4*(F3) $);

      \path[draw,trj,color=_mage] (z0-) -- (z0+);
      \path[draw,trj,color=_mage] (z2-) -- (z2+);
      \path[draw,trj,color=_mage] (z3-) -- (z3+);
      \path[draw,->,trj,color=_mage] (z0-) -- ($ (z0-) + .4*(F0) $);
      \path[draw,->,trj,color=_mage] (z2-) -- ($ (z2-) + .4*(F2) $);
      \path[draw,->,trj,color=_mage] (z3-) -- ($ (z3-) + .3*(F3) $);

      %\node[txt,rotate=40] at ($.5*(x0+) + .5*(x3-) + (-0.3,0.1)$) {$R_{(0,3)}$};
      \path[draw,->,rst] (x0+) -- (x3-);

      \node[txt] at ($.5*(y0+) + .5*(y1-) + (-0.99,0)$) {$R_{(0,1)}$};
      \path[draw,->,rst] (y0+) .. controls ($.75*(y0+) + .25*(y1-) + (-.5,0)$)
      and ($.25*(y0+) + .75*(y1-) + (-.5,0)$) .. (y1-);

      \node[txt] at ($.5*(y1+) + .5*(y3-) + (0,+.7)$) {$R_{(1,3)}$};
      \path[draw,->,rst] (y1+) .. controls ($.75*(y1+) + .25*(y3-) + (0,+.5)$)
      and ($.25*(y1+) + .75*(y3-) + (0,+.5)$) .. (y3-);

      \node[txt] at ($.5*(z0+) + .5*(z2-) + (0,-.7)$) {$R_{(0,2)}$};
      \path[draw,->,rst] (z0+) .. controls ($.75*(z0+) + .25*(z2-) + (0,-.5)$)
      and ($.25*(z0+) + .75*(z2-) + (0,-.5)$) .. (z2-);

      \node[txt] at ($.5*(z2+) + .5*(z3-) + (1.,0)$) {$R_{(2,3)}$};
      \path[draw,->,rst] (z2+) .. controls ($.75*(z2+) + .25*(z3-) + (+.5,0)$)
      and ($.25*(z2+) + .75*(z3-) + (+.5,0)$) .. (z3-);
    \end{tikzpicture}
  }%
  \hfill%
  \mbox{}%
  \caption{%
    Illustration of digital control system governing the time evolution of two physical quantities $(\xi_1,\xi_2) \in \R^2$ (Fig.~\ref{fig:dcs}) and a controlled hybrid system induced by discrete transitions in the digital controller state (Fig.~\ref{fig:chs}).
    The execution $x$ undergoes two discrete transitions simultaneously; the nearby executions $y_\delta$, $z_\delta$ encounter different discrete transition sequences.
    Since $R_{(1,3)}\circ R_{(0,1)} (0,0) = R_{(2,3)}\circ R_{(0,2)} (0,0) = (0,0)$, either transition sequence may be chosen for $x$.
  }
  \label{fig:motivation}
\end{figure}
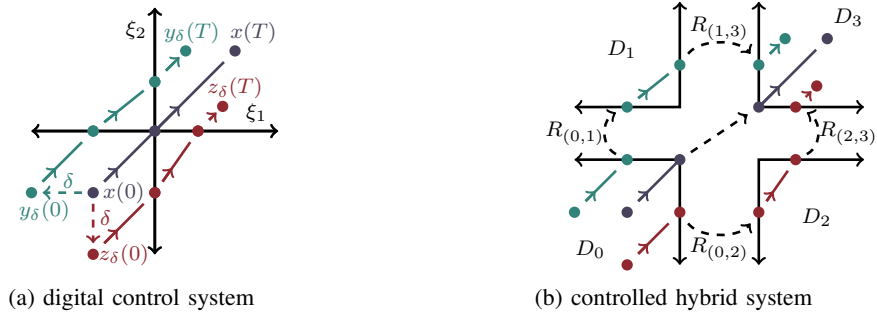

% For continuous--state dynamical systems and finite--state automata there separately exist rich sets of tools for metrization and simulation.
% The interaction of discrete transitions with continuous dynamics that is characteristic of controlled hybrid systems requires a new approach.
% To address this limitation, in this paper
To overcome this limitation, we construct a distance metric over the state space of a controlled hybrid system and apply this metric to develop a provably--convergent numerical simulation algorithm applicable to the class of hybrid systems illustrated in Fig.~\ref{fig:chs}.
Our framework enables formal investigation of a wide range of systems:
the dynamics may be nonlinear, the continuous dynamics may be controlled, and multiple discrete transitions may occur simultaneously, so long as executions are orbitally stable.

% Efforts to topologize and subsequently metrize controlled hybrid systems have been significant but fragmented.
Efforts to construct topologies on controlled hybrid systems have been significant, and can be best appreciated in this context by determining whether they induce a metric space.
Nerode and Kohn~\cite{Nerode1993} define state--space topologies that are not required to be metric spaces but are generated by finite--state automata associated with digital control systems.
Simic~\etal~\cite{Simic2005} apply a quotient construction to obtain, under regularity conditions, a topological manifold (or \emph{hybrifold}).
Ames and Sastry~\cite{Ames2005b} derive a category--theoretic \emph{colimit} topology over the \emph{regularization} proposed by Johansson~\etal~\cite{JohanssonEgerstedt1999} that relaxes domains at the guard sets.
We propose a metric topology over the state space of controlled hybrid systems that connects disparate domains through the reset map,
effectively metrizing the \emph{hybrifold} and \emph{colimit} topologies, and generalizing the \emph{phase space} metric proposed by Schatzman for an impact oscillator~\cite{Schatzman1998}.
% sam: it pains me to add another Teel ref, but this is where (T,J,eps)-close is first defined . . .
In contrast, Tavernini~\cite{Tavernini1987} and Gobel and Teel~\cite{GoebelTeel2006} directly metrize the space of executions of hybrid systems; Gokhman~\cite{Gokhman2008} demonstrates the equivalence of the resulting topology with that generated by the Skorohod trajectory metric~\cite[Chapter~6]{Pollard1984}).
We highlight in more detail the limitations imposed by metrizing the trajectory space rather than state space in Section~\ref{sec:metric}.

% The notion of %\emph{regularization} or
% \emph{relaxation} of a hybrid system should not be confused with the ``relaxation'' of hybrid inclusions described by Cai~\etal~\cite{CaiGoebel2008rx}.
% Since interpreting our controlled hybrid systems as hybrid inclusions yields singleton--valued ``flow'' and ``jump'' maps, relaxation in this sense does not yield a distinct hybrid system.
% sam-3.3,3.4
% Sanfelice and Teel~\cite{SanfeliceTeel2010} subsequently prove existence of approximating executions for a given ``simulation'' of a hybrid inclusion.
% In this paper we consider the opposite problem of proving convergence of approximating simulations for a given execution of a controlled hybrid system.

The literature on numerical simulation of hybrid systems may be partitioned into two groups: practical algorithms focused on high--precision estimates of discrete event times, and theoretical proofs of convergence for simulations. % of certain classes of hybrid systems.
Practical algorithms aim to place time--steps close to discrete event times using root--finding~\cite{Carver1978_,Shampine1991,GuckenheimerNerode1992}.
Theoretical proofs of convergence have generally required restrictive assumptions:
Esposito~\etal~\cite{EspositoKumar2001}, apply feedback linearization to asymptotically guarantee event detection for semi--algebraic guards, while Paoli and Schatzman~\cite{PaoliSchatzman2003ii} develop a provably--convergent algorithm for mechanical states undergoing impact. % with a unilateral constraint.
%
% sam-3.3,3.4
The most general convergence results relax the requirement that discrete transition times be determined accurately~\cite{Tavernini1987,Tavernini2009,SanfeliceTeel2010,BurdenGonzalezVasudevan2011}, and consequently can accommodate arbitrary nonlinear transition surfaces, Lipschitz continuous vector fields, and continuous discrete transition maps.
We extend this approach using our state--space metric to prove convergence, at least at a linear rate, to executions that satisfy an \emph{orbital stability} property described in Section~\ref{sec:simulation}.
Our algorithm is applicable to hybrid systems possessing control inputs and overlapping guards, representing a substantial contribution beyond our previous efforts~\cite{BurdenGonzalezVasudevan2011} and those of others~\cite{Tavernini1987, Tavernini2009, SanfeliceTeel2010}.

%\paragraph*{Motivating example}
%Consider the time evolution of two physical quantities $(\xi_1,\xi_2)\in\R^2$ monitored by a digital control system.
%Assume that when either quantity crosses a specified threshold, a logical controller state changes, causing a discontinuous change in the time derivatives $(\dot{\xi}_1,\dot{\xi}_2)$.
%Neglecting the discrete controller component of the system's state evolution yields a discontinuous planar vector field as in Fig.~\ref{fig:dcs}.
%A faithful model of the system's full state evolution is \emph{hybrid}, representing both discrete and continuous state transitions, as in Fig.~\ref{fig:chs}.
%As we show in Section~\ref{sec:metric}, state--of--the--art simulation algorithms for controlled hybrid systems based on trajectory--space metrics fail to simulate executions that undergo simultaneous discrete transitions (e.g. $x$) since nearby executions encounter different discrete transition sequences (e.g. $y_\delta$, $z_\delta$).
%To overcome this limitation, we introduce a novel state--space metric that enables derivation of a provably--convergent simulation algorithm applicable to the class of hybrid systems illustrated in Fig.~\ref{fig:chs}.

\paragraph*{Organization}
Section~\ref{sec:problem_formulation} contains definitions of mathematical concepts of interest.
Section~\ref{sec:relaxation} contains our technique for metrization and relaxation of controlled hybrid systems.
In Section~\ref{sec:simulation} we develop an algorithm for numerical simulation and prove uniform convergence at a linear rate of simulations to orbitally stable executions.
The technical and practical advantages of our techniques are illustrated in a series of examples in Section~\ref{sec:examples}.

% sam:
%   1. when referring to hybrifold or colimit, i say our construction is 'a metrization'
%   2. the phrasing in definitions, theorems, & figures must be parallel for M & \rx{M}
\section{Preliminaries}
\label{sec:problem_formulation}

We begin with the definitions and assumptions used throughout the paper.

\subsection{Topology}

The $2$--norm is our finite--dimensional norm of choice unless otherwise specified.
Let $P_A$ be the set of all finite partitions of $A \subset \R$.
Given $n \in \N$, we define the \emph{total variation} of $f \in L^\infty(\R,\R^n)$ by:
\begin{equation}
  \label{eq:BV_norm}
  \totvar{f}
  \!=\! \sup \setbb{\sum_{j=0}^{m-1} \norm{f(t_{j+1}) \!-\! f(t_{j})}_1 \!\mid \!\set{t_k}_{k=0}^m \!\in P_{\R},
    m \in \N},
\end{equation}
where $L^\infty(\R,\R^n)$ is the set of all almost everywhere bounded functions from $\R$ to $\R^n$.
The total variation of $f$ is a semi--norm, i.e.\ it satisfies the Triangle Inequality, but does not separate points.
$f$ is of \emph{bounded variation} if $\totvar{f} < \infty$, and we define $BV(\R,\R^n)$ to be the set of all functions of bounded variation from $\R$ to $\R^n$.

Given $n \in \N$ and $D \subset \R^n$, $\bdry{D}$ is the boundary of $D$, and $\setint{D}$ is the interior of $D$.
Recall that given a collection of sets $\set{S_\alpha}_{\alpha \in {\cal A}}$, where ${\cal A}$ might be uncountable, the \emph{disjoint union} of this collection is $\coprod_{\alpha \in {\cal A}} S_\alpha = \bigcup_{\alpha \in {\cal A}} S_\alpha \times \set{\alpha}$, a set that is endowed with the piecewise--defined topology.
Throughout the paper we will abuse notation and say that given $\bar{\alpha} \in {\cal A}$ and $x \in S_{\bar{\alpha}}$, then $x \in \coprod_{\alpha \in {\cal A}} S_\alpha$, even though we should write $\iota_{\bar{\alpha}}(x) \in \coprod_{\alpha \in {\cal A}} S_\alpha$, where $\iota_{\bar{\alpha}}\colon S_{\bar{\alpha}} \to \coprod_{\alpha \in {\cal A}} S_\alpha$ is the \emph{canonical identification} $\iota_{\bar{\alpha}}(x) = (x,\bar{\alpha})$.

In this paper we make extensive use of the concept of a \emph{quotient topology} induced by an equivalence relation defined on a topological space.
We regard a detailed exposition of this important concept as outside the scope of this paper, and refer the reader to Chapter~3 in~\cite{Kelley1955} or Section~22 in~\cite{Munkres2000} for more details.
The next definition formalizes equivalence relations in topological spaces induced by functions.
If $f\colon A \to B$, $V \subset A$, and $V' \subset B$, then we let $f(V) = \set{f(a) \in B \mid a \in V}$ denote the image of $V$ under $f$, and $f^{-1}(V') = \set{a \in A \mid f(a) \in V'}$ denote the pre--image of $V$ under $f$.
\begin{definition}
  \label{def:equiv_relation}
  Let ${\cal S}$ be a topological space, $A, B \subset {\cal S}$ two subsets, and $f\colon A \to B$ a function.
  %% HG: 1.3 and 3.2
  The \emph{$f$--induced equivalence relation}, denoted $\Lambda_f$, is the smallest equivalence relation containing the set $\setb{(a,b) \in {\cal S} \times {\cal S} \mid a \in f^{-1}(b)}$ (Section~4.2.4 in~\cite{Hein2009}).
  We say that \emph{$a, b \in {\cal S}$ are $f$--related}, denoted by $a \simfunc{f} b$, if $(a,b) \in \Lambda_f$.
  Moreover, the \emph{equivalence class of $x \in {\cal S}$} is defined as $\equivclass{x}{f} = \setB{a \in {\cal S} \mid a \simfunc{f} x}$, and the set of equivalence classes is defined as $\frac{\cal S}{\Lambda_f} = \setb{\equivclass{x}{f} \mid x\in {\cal S}}$.
  We endow the \emph{quotient} $\frac{\cal S}{\Lambda_f}$ with the quotient topology.
\end{definition}
\noindent Note that $\Lambda_f$ is reflexive, symmetric, and transitive, i.e.\ an equivalence relation.
An important application of the function--induced quotient is the construction of a single topological space out of several disconnected sets.
Indeed, given a collection of sets $\set{S_\alpha}_{\alpha \in {\cal A}}$, where ${\cal A}$ is some index set, and a function $f\colon U \to \coprod_{\alpha \in {\cal A}} S_\alpha$, where $U \subset \coprod_{\alpha \in {\cal A}} S_\alpha$, then $\widehat{S} = \frac{\coprod_{\alpha \in {\cal A}} S_\alpha}{\Lambda_f}$ is a topological space.
%Moreover, if $U \cap S_\alpha \neq \emptyset$ for each $\alpha \in {\cal A}$, then $\widehat{S}$ is a connected topological space. \sam{this is false}

Next, we present a useful concept from graph theory that simplifies our ensuing analysis.
\begin{definition}
  \label{def:graph_nbhd}
  Let $({\cal J},\Gamma)$ be a directed graph, where ${\cal J}$ is the set of vertices and $\Gamma \subset {\cal J} \times {\cal J}$ is the set of edges.
  Then, given $j \in {\cal J}$, define the \emph{neighborhood of $j$}, denoted ${\cal N}_j$, by:
  \begin{equation}
    \label{eq:graph_nbhd}
    {\cal N}_j = \set{e \in \Gamma \mid \exists j' \in {\cal J}\ \text{such that}\ e = (j,j')}.
  \end{equation}
\end{definition}

%%%%%%%%%%%%%%%%%%%%%%%%%%%%%%%%%%%%%
%%%%%%%%%%%%%%%%%%%%%%%%%%%%%%%%%%%%%
\subsection{Length Metrics}

Every metric space has an induced length metric, defined by measuring the length of the shortest curve between two points.
Throughout this paper, we use induced length metrics to metrize the function--induced quotients of disjoint unions of sets.
%, hence we will take a nontraditional approach for their definition. \sam{i was confused by what 'tradition' we were violating so i cut this}
To formalize this approach, we begin by defining the length of a curve in a metric space; the following definition is equivalent to Definition~2.3.1 in~\cite{Burago2001}.
\begin{definition}
  \label{def:curve_length}
  Let $(S,d)$ be a metric space, $I \subset [0,1]$ be an interval, and $\gamma\colon I \to {\cal S}$ be a continuous function.
  Define the \emph{length of $\gamma$ under the metric $d$} by:
  \begin{equation}
    \label{eq:curve_length}
    L_d(\gamma) = \sup\setbb{\sum_{i=0}^{k-1} d\bigl( \gamma(\bar{t}_i), \gamma(\bar{t}_{i+1}) \bigr)
      \mid k \in \N,\ \set{\bar{t}_i}_{i=0}^k \in P_I}.
  \end{equation}
\end{definition}
% Note that the supremum in the above definition is taken with respect to arbitrary sequences $\set{t_i}_{i=0}^k\subset I$. \humberto{this sentence is useless, it doesn't help understanding *why* you need arbitrary sequences, which is the complain of Ref5.}

We now define a generalization of continuous curves for quotiented disjoint unions.
\begin{definition}
  \label{def:connected_curve}
  Let $\set{S_\alpha}_{\alpha \in {\cal A}}$ be a collection of sets and $f\colon U \to \coprod_{\alpha \in {\cal A}} S_\alpha$, where $U \subset \coprod_{\alpha \in {\cal A}} S_\alpha$.
  $\gamma\colon [0,1] \to \coprod_{\alpha \in {\cal A}} S_\alpha$ is \emph{$f$--connected} if there exists
  $k \in \N$ and % \humberto{we need to define k somewhere, so we define it here.}
  $\set{t_i}_{i=0}^k \subset [0,1]$ with $0 = t_0 \leq t_1 \leq \ldots \leq t_k = 1$ such that $\gamma|_{[t_i,t_{i+1})}$ is continuous for each $i \in \set{0,1,\ldots, k-2}$, $\gamma|_{[t_{k-1},t_{k}]}$ is continuous, and $\lim_{t \uparrow t_i} \gamma(t) \simfunc{f} \gamma(t_i)$ for each $i \in \set{0,1,\ldots, k-1}$.
  Moreover, in that case $\set{t_i}_{i=0}^k$ is called a \emph{partition of $\gamma$}.
\end{definition}
\noindent Note that, since each section $\gamma|_{[t_i,t_{i+1})}$ is continuous, it must necessarily belong to a single set $S_\alpha$ for some $\alpha \in {\cal A}$ because the disjoint union is endowed with the piecewise--defined topology.
% \sam{i don't see why this is true -- i think it should be an extra hypothesis in the definition.}
% \humberto{what is the problem? the piecewise topology? if so, we can add another reference to Kelley.}
In the case when ${\cal A} = \set{\alpha}$ is a singleton, then every $\id_{S_\alpha}$--connected curve is simply a continuous curve over $S_\alpha$, where $\id_{S_\alpha}$ denotes the identity function in $S_\alpha$.
% Letting $\id$ denote the identity function, note that if ${\cal A}$ contains only one element, then every $\id$--connected curve is simply a continuous curve over the original set.
Figure~\ref{fig:connected_curve} shows an example of a connected curve over a collection of two sets.

\begin{figure}[t]
  \centering
  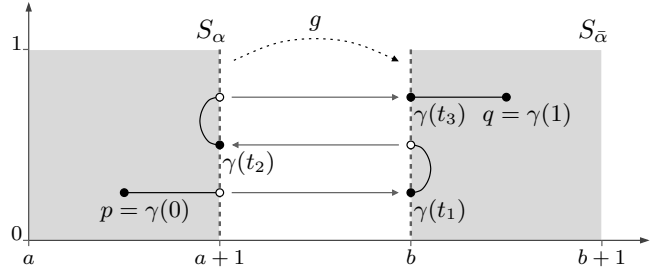%
  \caption{$g$--connected curve $\gamma$ with partition $\set{t_i}_{i=0}^4$, where $S_\alpha = {[a,a+1]} \times {[0,1]}$, $S_{\bar{\alpha}} = {[b,b+1]} \times {[0,1]}$, and $g \colon \set{a+1} \times {[0,1]} \to \set{b} \times {[0,1]}$ with $g(a+1,x) = (b,x)$.}%
  \label{fig:connected_curve}%
\end{figure}

Using the concept of connected curves, we now define the induced length distance of a collection of metric spaces.
The induced length distance is a generalization of the induced metric defined in Chapter~2 in~\cite{Burago2001}.
\begin{definition}
  \label{def:induced_metric}
  Let $\set{(S_\alpha,d_\alpha)}_{\alpha \in {\cal A}}$ be a collection of metric spaces, and let $\set{X_\alpha}_{\alpha \in {\cal A}}$ be a collection of sets such that $X_\alpha \subset S_\alpha$ for each $\alpha \in {\cal A}$.
  Furthermore, let $f\colon U \to \coprod_{\alpha \in {\cal A}} X_\alpha$, where $U \subset \coprod_{\alpha \in {\cal A}} X_\alpha$, and let $\widehat{X} = \frac{\coprod_{\alpha \in {\cal A}} X_\alpha}{\Lambda_f}$.
 $\indmetric{\widehat{X}}\colon \widehat{X} \times \widehat{X} \to [0,\infty]$ is the \emph{$f$--induced length distance of $\widehat{X}$}, defined by:
  \begin{multline}
    \label{eq:induced_metric}
    \indmetric{\widehat{X}}(p,q)
    \!=\! \inf\setbb{\sum_{i=0}^{k-1} L_{d_{\alpha_i}}\parenb{\gamma|_{[t_i,t_{i+1})}} \mid{}
      \gamma\colon [0,1] \to \coprod_{\alpha \in {\cal A}} X_\alpha,
      \gamma(0) = p,\ \gamma(1) = q,\\
      \gamma\ \text{is $f$--connected},\ \set{t_i}_{i=0}^k \in P_{[0,1]},\
      \set{\alpha_i}_{i=0}^{k-1}\ \text{s.t.}\ \gamma\bigl([t_i,t_{i+1})\bigr) \subset X_{\alpha_i}\ \forall i}.
  \end{multline}
\end{definition}
\noindent
%We use induced length distances in two situations: first, to define a distance over subsets of a single metric space, and second, to define a distance over disjoint unions of metric spaces.
We invoke this definition to metrize both subsets and disjoint unions of metric spaces.
%Definition~\ref{def:induced_metric} addresses both types of distance functions in a single compact form.
%Although induced length distances are non--negative, symmetric, and subadditive, they do not separate points in general (Section~2.3 in~\cite{Burago2001}), i.e., every induced length distance is a pseudometric, but not necessarily a metric.
It is important to note that although $\indmetric{\widehat{X}}$ is non--negative, symmetric, and subadditive, it does not necessarily separate points of $\widehat{X}$ \see{Section~2.3 in~\cite{Burago2001}}, and hence generally only defines a \emph{pseudo}--metric.
In the special case where no function $f$ is supplied, then by convention we let $f = \id_X$, the identity function on $X$.
This implies $\widehat{X} = X$ and the induced metric coincides with the given metric.
The following Lemma is a straightforward consequence of Proposition~2.3.12 in~\cite{Burago2001}.
%The following is an important particular case of induced length distance. The proof can be found in Section~2.3 in~\cite{Burago2001} and is omitted here since it falls outside the scope of this paper.
\begin{lemma}\label{lem:induced}
  Let $(S,d)$ be a metric space and $X \subset S$.
  %Then $\indmetric{X}$, the $\id$--induced length distance on $X$, is a metric.
  Then $\indmetric{X}$ is a metric.
  Moreover, the topology on $X$ induced by $\indmetric{X}$ is equivalent to the topology on $X$ induced by $d$.
\end{lemma}

%%%%%%%%%%%%%%%%%%%%%%%%%%%%%%%%%%%%%
%%%%%%%%%%%%%%%%%%%%%%%%%%%%%%%%%%%%%
\subsection{Controlled Hybrid Systems}
\label{sec:hybrid_systems}

Motivated by the definition of hybrid systems presented in~\cite{Simic2005}, we define the class of hybrid systems of interest in this paper.

\begin{definition}
  \label{def:hybrid_system}
  A \emph{controlled hybrid system} is a tuple
  \begin{equation}
    {\cal H} = ({\cal J}, \Gamma, {\cal D}, U, {\cal F}, {\cal G}, {\cal R}),
  \end{equation}
  where:
  \begin{newitemize}
  \item ${\cal J}$ is a finite set indexing the discrete states of ${\cal H}$;
  \item $\Gamma \subset {\cal J} \times {\cal J}$ is the set of edges, forming a directed graph structure over ${\cal J}$;
  \item ${\cal D} = \set{D_j}_{j \in {\cal J}}$ is the set of domains, where each $D_j$ is a subset of $\R^{n_j}$, $n_j \in \N$;
  \item $U\subset\R^m$ is the range space of control inputs, $m\in\N$;
  \item ${\cal F} = \set{f_j}_{j \in {\cal J}}$ is the set of vector fields, where each $f_j\colon \R \times D_j \times U \to \R^{n_j}$ is the vector field defining the dynamics of the system on $D_j$;
  \item ${\cal G} = \set{G_e}_{e \in \Gamma}$ is the set of guards, where each $G_{(j,j')} \subset \bdry{D_j}$ is a guard in mode $j \in {\cal J}$ that defines a transition to mode $j' \in {\cal J}$; and,
  \item ${\cal R} = \set{R_e}_{e \in \Gamma}$ is the set of reset maps, where each map $R_{(j,j')}\colon G_{(j,j')} \to D_{j'}$ defines the transition from guard $G_{(j,j')}$.
  \end{newitemize}
\end{definition}
\noindent For convenience, we sometimes refer to controlled hybrid systems as just hybrid systems, and we refer to the distinct vertices within the graph structure associated with a controlled hybrid system as modes.
Each domain in the definition of a controlled hybrid system is a metric space with the Euclidean distance metric.
A three--mode autonomous hybrid system, which is a particular case of Definition~\ref{def:hybrid_system} where none the vector fields $\set{f_j}_{j \in {\cal J}}$ depend on the control input, is illustrated in Fig.~\ref{fig:hybrid_system}.
Note that we restrict control inputs to the continuous flow, hence inputs do not have an effect during discrete transitions.

Next, we impose several technical assumptions that support existence and uniqueness of executions on hybrid domains.
We delay the definition of executions of a hybrid system to Section~\ref{subsec:hybrid_execution} once all the technical details regarding the metrization of spaces have been presented.

\begin{figure}
  \centering
  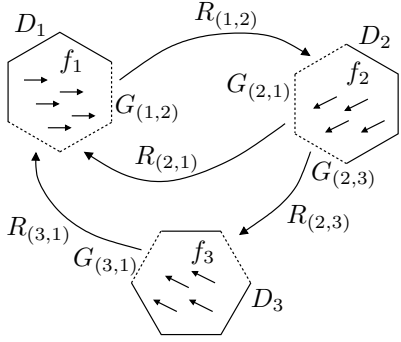%
  \caption{Illustration of a controlled hybrid system with three modes.}
  \label{fig:hybrid_system}%
\end{figure}

% sam-1.10: moved this set of assumptions
\begin{assumption}
  \label{assump:topology}
  Let ${\cal H}$ be a controlled hybrid system.
  Then the following statements are true:
  \begin{enumparen}
  \item\label{assump:top_dom_manifolds} For each $j \in {\cal J}$, $D_j$ is a compact $n_j$--dimensional manifold with boundary.
  \item\label{assump:top_control} $U$ is compact.
  \item\label{assump:top_guards_submanifolds} For each $e \in \Gamma$, $G_e$ is a closed, embedded, codimension $1$ submanifold with boundary.
  \item\label{assump:top_reset_cont} For each $e \in \Gamma$, $R_e$ is continuous.
  \end{enumparen}
\end{assumption}

\begin{assumption}
  \label{assump:lipschitz_vf}
  For each $j \in {\cal J}$, $f_j$ is Lipschitz continuous.
  That is, there exists $L > 0$ such that for each $j \in {\cal J}$, $t_1, t_2 \in \R$, $x_1, x_2 \in D_j$, and $u_1, u_2 \in U$:
  \begin{equation}
    \label{eq:lipschitz_vf}
    \normb{f_j( t_1, x_1, u_1) - f_j( t_2, x_2, u_2)}
    \leq L \bigl( \abs{t_1 - t_2} + \norm{x_1 - x_2} + \norm{u_1 - u_2} \bigr).
  \end{equation}
\end{assumption}
\noindent Assumption~\ref{assump:lipschitz_vf} guarantees the existence and uniqueness of solutions to ordinary differential equations in individual domains.
%% HG: if we add this note, we have to make the same remark for the state as well.
% Note that the elements $u_j$ in the assumption are individual input vectors, not control signals.
In the sequel we will consider control inputs of \emph{bounded variation} $u \in BV(\R,U)$.
Note that without loss of generality we take $0$ as the initial time in the following Lemma; a general initial time can be accommodated by a straightforward change of variables.
\begin{lemma}
  \label{lemma:exist_and_unique}
  Let ${\cal H}$ be a controlled hybrid system.
  Then for each $j \in {\cal J}$, each initial condition $p \in D_j$, and each control $u \in BV(\R,U)$, there exists an interval $I \subset \R$ with $0 \in I$
  such that the following differential equation has a unique solution:
  \begin{equation}
    \label{eq:ode}
    \dot{x}(t) = f_j\bigl( t, x(t), u(t) \bigr),\quad t \in I,\quad x(0) = p.
  \end{equation}
  $x$ is called the \emph{integral curve} of $f_j$ with initial condition $p$ and control $u$.
  %% HG: Response to Ref7 Q2.5
  Moreover, $x|_I$ is absolutely continuous.
\end{lemma}
% sam-1.10
\begin{proof}
  Let $\widetilde{f}_j\colon \R \times \R^{n_j} \times U \to \R^{n_j}$ be any globally Lipschitz continuous extension to $f_j$ (guaranteed to exist by Theorem~1 in~\cite{McShane1934}).
  Given any $p \in D_j \subset \R^{n_j}$ and $u \in BV(\R,U)$, Proposition~5.6.5 in~\cite{Polak1997} guarantees the existence of an integral curve $\widetilde{x}\colon \widetilde{I} \to \R^{n_j}$ for $\widetilde{f}_j$ with initial condition $\widetilde{x}(0) = p$.
  Note that $\widetilde{x}$ is absolutely continuous by Theorem~3.35 in~\cite{Folland1999}.
  Let $I \subset \widetilde{I}$ be the connected component of $\widetilde{x}^{-1}(D_j)$ containing $0$.
  Then $x = \widetilde{x}|_I$ is an absolutely continuous integral curve of $f_j$ and $x(I) \subset D_j$.
  Note that $x$ is unaffected by the choice of extension $\widetilde{f}_j$.
%  \sam{Just because $D_j$ is a compact $n_j$--dimensional manifold doesn't mean it admits an embedding into $\R^{n_j}$; need to introduce the embedding into $\R^{2n_j+1}$.}
%  Let $\iota:D_j\hookrightarrow\R^{2n_j+1}$ be a smooth embedding (guaranteed to exist by~\cite[Theorem~10.11]{Lee2003}) and
%  let $\widetilde{f}_j\colon \R \times \R^{2n_j+1} \times U \to \R^{2n_j+1}$ be any globally Lipschitz continuous extension to $D\iota\circ f_j$ (guaranteed to exist by~\cite[Theorem~1]{McShane1934}).
%  Given any $p \in \iota(D_j) \subset \R^{2n_j+1}$ and $u \in BV(\R,U)$, Proposition~5.6.5 in~\cite{Polak1997} guarantees the existence of an integral curve $\widetilde{x}\colon \widetilde{I} \to \R^{2n_j+1}$ for $\widetilde{f}_j$ with initial condition $\widetilde{x}(0) = p$.
%  Note that $\widetilde{x}$ is absolutely continuous by Theorem~3.35 in~\cite{Folland1999}.
%  Let $I \subset \widetilde{I}$ be the connected component of $\widetilde{x}^{-1}(\iota(D_j))$ containing $0$.
%  Then $x = \iota^{-1}\circ\widetilde{x}|_I$ is an absolutely continuous integral curve of $f_j$, $x(0) = \iota^{-1}(p)$, and $x(I) \subset D_j$.
%  Note that $x$ is unaffected by the choice of extension $\widetilde{f}_j$ or embedding $\iota$.
\end{proof}

The following definition is used to construct executions of a controlled hybrid system.
\begin{definition}
  \label{def:maximal_curve}
  Let ${\cal H}$ be a controlled hybrid system, $j \in {\cal J}$, $p \in D_j$, and $u \in BV(\R,U)$.
  $x\colon I \to D_j$ is the \emph{maximal integral curve} of $f_j$ with initial condition $p$ and control $u$ if, given any other integral curve with initial condition $p$ and control $u$, such as $\tilde{x}\colon \widetilde{I} \to D_j$, then $\widetilde{I} \subset I$.
\end{definition}

% sam-1.10
\noindent Given a maximal integral curve $x\colon I \to D_j$, a direct consequence\footnote{This follows from continuity of integral curves and compactness of hybrid domains.} of Definition~\ref{def:maximal_curve} and Assumption~\ref{assump:topology} is that either $\sup I = +\infty$, or $\sup I = t' < \infty$ and $x(t') \in \bdry{D_j}$.
%\sam{we need a reference for this;  Theorem 3.22 in~\cite{Sastry1999} would get us partly there (though still need to prove $x(t')\in\bdry{D_j}$) but requires $u$ be piecewise--continuous}
This fact is critical during the definition of executions of a controlled hybrid systems in Section~\ref{sec:simulation}.

\section{Metrization and Relaxation of Controlled Hybrid Systems}
\label{sec:relaxation}

In this section, we metrize a unified family of spaces containing all the domains of a controlled hybrid system ${\cal H}$.
The constructed metric space has three appealing properties:
first, the distance between a point in a guard and its image via its respective reset map is zero;
second, the distance between points in different domains are properly defined and finite;
and third, the distance between points is based on the Euclidean distance metric from each domain.
% \begin{enumparen}
% \item the distance between a point in a guard and its image via its respective reset map is zero;
% \item the distance between points in different domains are properly defined and finite; and,
% \item the distance between points is based on the Euclidean distance metric from each domain.
% \end{enumparen}

\subsection{Hybrid Quotient Space}

Using Definitions~\ref{def:induced_metric} and~\ref{def:hybrid_system}, we construct a metric space where the executions of a controlled hybrid system reside.
The result is a metrization of the \emph{hybrifold}~\cite{Simic2005}.
\begin{definition}
  \label{def:hybrid_quotient_space}
  Let ${\cal H}$ be a controlled hybrid system, and let
  \begin{equation}
    \label{eq:hatR}
    \widehat{R}\colon \coprod_{e \in \Gamma} G_e \to \coprod_{j \in {\cal J}} D_j
  \end{equation}
  be defined by $\widehat{R}(p) = R_e(p)$ for each $p \in G_e$. Then the \emph{hybrid quotient space} of $\cal H$ is:
  \begin{equation}
    \label{eq:hybrid_quotient_space}
    {\cal M} = \frac{\coprod_{j \in {\cal J}} D_j}{\Lambda_{\widehat{R}}}.
  \end{equation}
  % where $\Lambda_\rx{R}$ is as in~\eqref{eq:equiv_relation}.
\end{definition}
\noindent Fig.~\ref{fig:hybrid_quotient_space} illustrates the construction in Definition~\ref{def:hybrid_quotient_space}.
\begin{figure}[t]
  \centering
  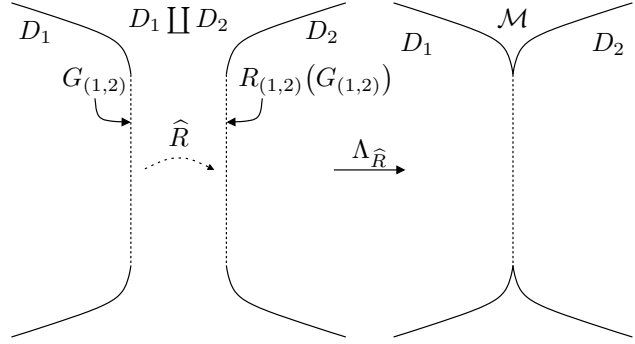%
  \caption{The disjoint union of $D_1$ and $D_2$ (left) and the hybrid quotient space ${\cal M}$ obtained from the relation $\Lambda_{\widehat{R}}$ (right).}
  \label{fig:hybrid_quotient_space}
\end{figure}
The induced length distance on ${\cal M}$ is in fact a distance metric:
\begin{theorem}
  \label{thm:metric_M}
  Let ${\cal H}$ be a controlled hybrid system, and let $\indmetric{\cal M}$ be the $\widehat{R}$--induced length distance of ${\cal M}$, where $\widehat{R}$ is defined in~\eqref{eq:hatR}.
Then $\indmetric{\cal M}$ is a metric on ${\cal M}$, and the topology it induces is equivalent to the $\widehat{R}$--induced quotient topology.
\end{theorem}
\begin{proof}
  We provide the main arguments of the proof, omitting the details in the interest of brevity.
  First, note that each domain is a normal space, i.e.\ every pair of disjoint closed sets have disjoint neighborhoods.
  Second, note that each reset map is a closed map, i.e.\ the image of closed sets under the reset map are closed.
  This fact follows by Condition~\ref{assump:top_guards_submanifolds} in Assumption~\ref{assump:topology}, since each guard is compact, thus reset maps are closed by the Closed Map Lemma (Lemma~A.19 in~\cite{Lee2003}).

  Let $\widehat{D} = \coprod_{j \in {\cal J}} D_j$ and $p,q \in \widehat{D}$.
  We aim to show that if $p$ and $q$ yield distinct equivalence classes (i.e. $(p,q)\notin\Lambda_{\hat{R}}$) then the induced distance between them is strictly positive.
  Note that the equivalence classes $\equivclass{p}{\widehat{R}}$ and $\equivclass{q}{\widehat{R}}$ are each a finite collection of closed sets.
  Moreover, since we can construct disjoint neighborhoods around each of these closed sets, then we can conclude that there exists $\delta > 0$ such that $\indmetric{\cal M}\parenb{\equivclass{p}{\widehat{R}}, \equivclass{q}{\widehat{R}}} > \delta$.
  The proof concludes by following the argument in Exercise~3.1.14 in~\cite{Burago2001}, i.e.\ since each connected component in $\widehat{D}$ is bounded, then ${\cal M}$ is also bounded (in the quotient topology).
  %% HG 1.16
  Then, using a simple extension of Theorem~5.8 in~\cite{Kelley1955}\footnote{The extension aims to allow the domain of the map to be bounded instead of compact. The new proof follows step--by--step the argument in~\cite{Kelley1955}.}, we get that the identity map from ${\cal M}$ to the space constructed by taking the quotient of all the points in $\widehat{D}$ such that $\indmetric{\cal M}$ has zero distance is a homeomorphism, thus they have the same topology.
\end{proof}
\noindent
% sam-1.17
It is crucial to note that all $\widehat{R}$--connected curves are continuous in the topology induced by the metric $\indmetric{\cal M}$ on the hybrid quotient space ${\cal M}$.
This implies in particular that executions of controlled hybrid systems (to be defined in Section~\ref{sec:simulation}) are continuous in ${\cal M}$ since the endpoint of the segment of an execution that lies in a guard $G_e$ will be $\widehat{R}$--related to the startpoint of the subsequent segment of the execution; alternately, this follows from Theorem~3.12(b) in~\cite{Simic2005} since ${\cal M}$ is equivalent to the ``hybrifold'' construction in that paper.
This important property is foundational to the convergence results for sequences of (relaxed) executions and their simulations derived in Section~\ref{sec:simulation}.
For further details, we refer the interested reader to Examples~3.2 and~3.3 in~\cite{Simic2005} where continuity is clearly discussed for simple examples.

\subsection{Relaxation of a Controlled Hybrid System}

% This section deals with the construction of the unified space mentioned above and its metric.
% First, we relax hybrid domains along their guards, allowing for slackness in the computation of the time instant when intersections occur between executions and a guard.
% Next, we attach the disparate domains to each other via a topological quotient and construct a single metric space.
To construct a numerical simulation scheme that does not require the exact computation of the time instant when an execution intersects a guard, we require a method capable of introducing some slackness within the computation. This is accomplished by relaxing\footnote{This should not be confused with the ``relaxation'' of hybrid inclusions described by Cai~\etal~\cite{CaiGoebel2008rx}.
Since interpreting our controlled hybrid systems as hybrid inclusions yields singleton--valued ``flow'' and ``jump'' maps, relaxation in this sense does not yield a distinct hybrid system.} each domain along its guard and then relaxing each vector field and reset map accordingly in order to define a relaxation of a controlled hybrid system.

To formalize this approach, we begin by defining the relaxation of each domain of a controlled hybrid system, which is accomplished by first attaching an $\rxsymbol$--sized strip to each guard.
% As we show in the next section, this eliminates the need to exactly detect guard intersection, simplifying numerical simulation.
\begin{definition}
  \label{def:relaxed_domain}
  Let ${\cal H}$ be a controlled hybrid system.
  For each $e \in \Gamma$, let $\rx{S_e} = G_e \times [0,\rxsymbol]$ be the \emph{strip associated to guard $G_e$}.
  For each $j \in {\cal J}$, let
  \begin{equation}
    \label{eq:iota_j}
    \chi_j\colon \coprod_{e \in {\cal N}_j} G_e \to \coprod_{e \in {\cal N}_j} \rx{S_e},
  \end{equation}
  be the \emph{canonical identification} of each point in a guard with its corresponding strip defined for each $p \in G_e$ as $\chi_j(p) = (p,0) \in \rx{S_e}$. Then, the \emph{relaxation of $D_j$} is defined by:
  \begin{equation}
    \label{eq:relaxed_domain}
    \rx{D_j} =
    \frac{D_j \coprod \paren{\coprod_{e \in {\cal N}_j} \rx{S_e}}}{\Lambda_{\chi_j}}.
  \end{equation}
  % where $\Lambda_{\chi_j}$ is defined as in~\eqref{eq:equiv_relation}.
\end{definition}
%% HG: I think Sam would want to make the point about the smoothness of the relaxed domain, but right now it takes us nowehere. Also, I have no idea if the relaxed domains are compact under the quotient topology.
% Note that, since $G_e$, for $e \in {\cal N}_j$, is a closed embedded submanifold of $\bdry{D_j}$ by Assumption~\ref{assump:closed_guards}, $\rx{S_e}$ is a compact smooth manifold.
\noindent By Condition~\ref{assump:top_guards_submanifolds} in Assumption~\ref{assump:topology}, each point on a strip $\rx{S_e}$ of $D_j$ is defined using $n_j$ coordinates $(\zeta_1,\ldots,\zeta_{n_j-1},\tau)$, shortened $(\zeta,\tau)$, where $\tau$ is called the \emph{transverse coordinate} and is the distance along the interval $[0,\rxsymbol]$.
An illustration of Definition~\ref{def:relaxed_domain} together with the coordinates on each strip is shown in Fig.~\ref{fig:relaxed_domain}.

\begin{figure}[t]
  \centering
  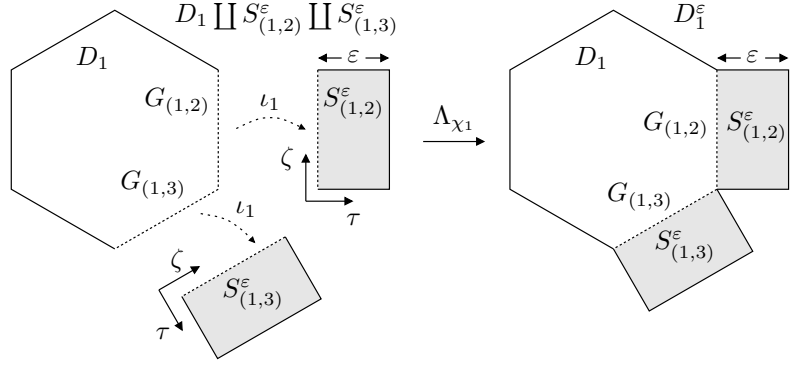%
  \caption{Disjoint union of $D_1$ and the strips in its neighborhood, $\set{\rx{S_e}}_{e \in {\cal N}_1}$ (left), and the relaxed domain $\rx{D_1}$ obtained from the relation $\Lambda_{\chi_1}$ (right).}%
  \label{fig:relaxed_domain}%
\end{figure}

We endow each $\rx{S_e}$ with a distance metric in order to define an induced length metric on a relaxed domain $\rx{D_j}$.
\begin{definition}
  \label{def:strip_metric}
  Let $j \in {\cal J}$ and $e \in {\cal N}_j$.
  % sam-1.15: added \zeta\in G_e and \tau\in[0,\veps]
  Endow $D_j$ with $\indmetric{D_j}$ as its metric, and $d_{\rx{S_e}}\colon \rx{S_e} \times \rx{S_e} \to [0,\infty)$ as the metric on $\rx{S_e}$, defined for each $\zeta,\zeta'\in G_e$ and $\tau,\tau'\in[0,\varepsilon]$ by:
  \begin{equation}
    \label{eq:strip_metric}
    d_{\rx{S_e}}\parenb{(\zeta,\tau), (\zeta',\tau')} = \indmetric{G_e}( \zeta, \zeta' ) + \abs{\tau - \tau'}.
  \end{equation}
\end{definition}

We now define a length metric on relaxed domains using Definitions~\ref{def:connected_curve} and~\ref{def:induced_metric}.
\begin{theorem}
  \label{thm:relaxed_metric_Dj}
  Let $j \in {\cal J}$, and let $\indmetric{\rx{D_j}}$ be the $\chi_j$--induced length distance on $\rx{D_j}$, where $\chi_j$ is as defined in~\eqref{eq:iota_j}.
  Then $\indmetric{\rx{D_j}}$ is a metric on $\rx{D_j}$, and the topology it induces is equivalent to the $\chi_j$--induced quotient topology.
\end{theorem}
\begin{proof}
  Since $\indmetric{\rx{D_j}}$ is non--negative, symmetric, and subadditive, it remains to show that it separates points.
  Let $p, q \in \rx{D_j}$.
  First, we want to show that $\equivclass{p}{\chi_j} = \equivclass{q}{\chi_j}$ whenever $\indmetric{\rx{D_j}}(p,q) = 0$.
  Note that for each $e \in {\cal N}_j$ and each pair $p,q \in G_e$, and by the Definition~\ref{def:induced_metric} and~\ref{def:strip_metric}, $d_\rx{S_e}( (p,0),(q,0)) \geq \indmetric{D_j}(p,q)$, hence no connected curve that transitions to a strip can be shorter than a curve that stays in $D_j$.
  This fact immediately shows that for $p,q \in D_j$, $\indmetric{\rx{D_j}}(p,q) = 0$ implies that $\equivclass{p}{\chi_j} = \equivclass{q}{\chi_j}$.
  The case when one of the points is in $G_e \times (0,\rxsymbol] \subset \rx{S_e}$ follows easily by noting that those points can be separated by a suitably--sized $d_\rx{S_e}$--ball.
  The proof concludes by following the argument in Exercise~3.1.14 in~\cite{Burago2001}, as we did in the proof of Theorem~\ref{thm:metric_M}.
\end{proof}
\noindent Refer to $\indmetric{\rx{D_j}}$ as the \emph{relaxed domain metric}.
Note that Theorem~\ref{thm:relaxed_metric_Dj} can be proved using essentially the same argument as in the proof of Theorem~\ref{thm:metric_M}, but we prove Theorem~\ref{thm:relaxed_metric_Dj} to emphasize the utility of the inequality relating the induced metric on a domain and the metric on each strip.

Next, we define a vector field over each relaxed domain.
\begin{definition}
  \label{def:relaxed_vf}
  Let $j \in {\cal J}$.
  For each $e \in {\cal N}_j$, let the vector field on the strip $\rx{S_e}$, denoted $f_e$, be the unit vector pointing outward along the transverse coordinate. In coordinates, $f_e\bigl( t, (\zeta,\tau), u \bigr) = \parenb{\underbrace{0,\ \dots,\ 0}_{\zeta\ \text{coords.}},\ 1}^T$.
  % \begin{equation}
  %   f_e\bigl( t, (\zeta,\tau), u \bigr) =
  %   \parenb{\underbrace{0,\ \dots,\ 0}_{\zeta\ \text{coords.}},\ 1}^T.
  % \end{equation}
  Then, the \emph{relaxation of $f_j$} is:
  \begin{equation}
    \label{eq:relaxed_vf}
    \rx{f_j}(t,x,u) =
    \begin{cases}
      f_j(t,x,u) & \text{if}\ x \in D_j, \\
      f_e(t,x,u) & \text{if}\ x \in G_e \times (0,\epsilon] \subset \rx{S_e},\ e \in {\cal N}_j. \\
    \end{cases}
  \end{equation}
\end{definition}
\noindent Note that the relaxation of the vector field is generally not continuous along each $G_e$, for $e \in {\cal N}_j$.
As we show in the algorithm in Fig.~\ref{fig:algo_relaxed_execution}, this discontinuous vector field does not lead to sliding modes on the guards~\cite{Utkin1977,Filippov1988}, since the vector field on the strips always points away from the guard.
An illustration of the relaxed vector field $\rx{f_j}$ on $\rx{D_j}$ is shown in Fig.~\ref{fig:relaxed_vf}.

\begin{figure}[t]
  \centering
  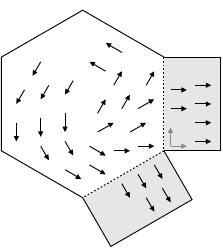%
  \caption{Relaxed vector field $\rx{f_1}$ on $\rx{D_1}$.}%
  \label{fig:relaxed_vf}%
\end{figure}

The definitions of relaxed domains and relaxed vector fields allow us to construct a relaxation of the controlled hybrid systems as follows:
\begin{definition}
  \label{def:relaxed_hybrid_system}
  Let ${\cal H}$ be a controlled hybrid system.
  We say that the relaxation of ${\cal H}$ is a tuple:
  \begin{equation}
    \rx{\cal H} = ({\cal J}, \Gamma, \rx{\cal D}, U, \rx{\cal F}, \rx{\cal G}, \rx{\cal R}),
  \end{equation}
  where:
  \begin{newitemize}
  \item $\rx{\cal D} = \setb{\rx{D_j}}_{j \in {\cal J}}$ is the set of relaxations of the domains in ${\cal D}$, and each $\rx{D_j}$ is endowed with its induced length distance metric $\indmetric{\rx{D_j}}$;
  \item $\rx{\cal F} = \setb{\rx{f_j}}_{j \in {\cal J}}$ is the set of relaxations of the vector fields in ${\cal F}$;
  \item $\rx{\cal G} = \setb{\rx{G_e}}_{e \in \Gamma}$ is the set of relaxations of the guards in ${\cal G}$, where $\rx{G_e} = G_e \times \set{\rxsymbol} \subset \rx{S_e}$ for each $e \in \Gamma$; and,
  \item $\rx{\cal R} = \setb{\rx{R_e}}_{e \in \Gamma}$ is the set of relaxations of the reset maps in ${\cal R}$, where $\rx{R_e}\colon \rx{G_e} \to D_{j'}$ for each $e = (j,j') \in \Gamma$ and $\rx{R_e}(\zeta,\rxsymbol) = R_e(\zeta)$ for each $\zeta \in G_e$.
  \end{newitemize}
\end{definition}

% \sam{remark that $\rx{\cal H}$ is not a hybrid dynamical system?}

\subsection{Relaxed Hybrid Quotient Space}

Analogous to the construction of the metric quotient space ${\cal M}$, using Definitions~\ref{def:induced_metric} and~\ref{def:relaxed_hybrid_system} we construct a unified metric space where executions of relaxations of controlled hybrid systems reside.
The result is a metrization of the \emph{hybrid colimit}~\cite{Ames2005b} (rather than a metrization of the hybrifold as in the previous section).
\begin{definition}
  \label{def:relaxed_space}
  Let $\rx{\cal H}$ be the relaxation of the controlled hybrid system ${\cal H}$.
  Also, let
  \begin{equation}
    \label{eq:rx_hatR}
    \rx{\widehat{R}}\colon \coprod_{e \in \Gamma} \rx{G_e} \to \coprod_{j \in {\cal J}} \rx{D_j}
  \end{equation}
  be defined by $\rx{\widehat{R}}(p) = \rx{R_e}(p)$ for each $p \in \rx{G_e}$.
  Then the \emph{relaxed hybrid quotient space} of $\rx{\cal H}$ is:
  \begin{equation}
    \label{eq:relaxed_space}
    \rx{\cal M} = \frac{\coprod_{j \in {\cal J}} \rx{D_j}}{\Lambda_\rx{\widehat{R}}}.
  \end{equation}
  % where $\Lambda_\rx{R}$ is as in~\eqref{eq:equiv_relation}.
\end{definition}
\noindent The construction in Definition~\ref{def:relaxed_space} is illustrated in Fig.~\ref{fig:relaxed_space}.

\begin{figure}[t]
  \centering
  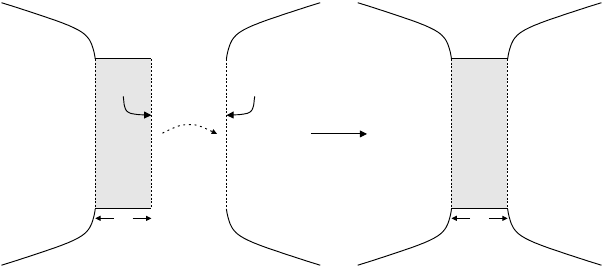%
  \caption{The disjoint union of $\rx{D_1}$ and $\rx{D_2}$ (left), and the relaxed hybrid quotient space $\rx{\cal M}$ obtained from the relation $\Lambda_{\rx{\widehat{R}}}$ (right).}
  \label{fig:relaxed_space}
\end{figure}

We now show that the induced length distance on $\rx{\cal M}$ is indeed a metric.
We omit this proof since it is identical to the proof of Theorem~\ref{thm:metric_M}.
\begin{theorem}
  \label{thm:relaxed_metric_M}
  Let ${\cal H}$ be a controlled hybrid system, let $\rx{\cal H}$ be its relaxation, and let $\indmetric{\rx{\cal M}}$ be the $\rx{\widehat{R}}$--induced length distance of $\rx{\cal M}$, where $\rx{\widehat{R}}$ is defined in~\eqref{eq:rx_hatR}.
  Then $\indmetric{\rx{\cal M}}$ is a metric on $\rx{\cal M}$, and the topology it induces is equivalent to the $\rx{\widehat{R}}$--induced quotient topology.
\end{theorem}
\noindent All $\rx{\widehat{R}}$--connected curves are continuous under the metric topology induced by $\indmetric{\rx{\cal M}}$ which will be important when we study executions of hybrid systems in Section~\ref{sec:simulation}.
%As we will see in the next section, this implies that under the $\indmetric{\rx{\cal M}}$--topology the executions of relaxed hybrid systems are continuous.

As expected, the metric on $\rx{\cal M}$ converges pointwise to the metric on ${\cal M}$.
\begin{theorem}
  \label{thm:metric_convergence}
  Let ${\cal H}$ be a controlled hybrid system, and let $\rx{\cal H}$ be its relaxation.
  Then for each $p,q \in {\cal M}$, $\indmetric{\rx{\cal M}}(p,q) \to \indmetric{\cal M}(p,q)$ as $\rxsymbol \to 0$.
\end{theorem}
\begin{proof}
  Abusing notation, let $L(\gamma)$ denote the length of any connected curve $\gamma$, defined as the sum of the lengths of each of its continuous sections under the appropriate metric.
  %\sam{Another abuse of notation: $\indmetric{\rx{\cal M}}(p,q)$.}
  First, note that $\indmetric{\cal M}(p,q) \leq \indmetric{\rx{\cal M}}(p,q)$.
  This inequality follows since, as we argued in the proof of Theorem~\ref{thm:relaxed_metric_Dj}, given an edge $(j,j') \in \Gamma$, $d_{\rx{S_{(j,j')}}}\parenb{(p',0),(q',0)} \geq \indmetric{D_j}(p',q')$ for any pair of points $p',q' \in G_{(j,j')}$.
  Thus, adding the strips $\set{S_e}_{e \in \Gamma}$ in $\rx{\cal M}$ only make the length of a connected curve longer.

  Now let $\widehat{D} = \coprod_{j \in {\cal J}} D_j$ and $\rx{\widehat{D}} = \coprod_{j \in {\cal J}} \rx{D_j}$.
  Given $\delta > 0$, there exists $\gamma\colon [0,1] \to \widehat{D}$, an $\widehat{R}$--connected curve with partition $\set{t_i}_{i=0}^k$, such that $\gamma(0) = p$, $\gamma(1) = q$, and $\indmetric{\cal M}(p,q) \leq L(\gamma) \leq \indmetric{\cal M}(p,q) + \delta$.
  Moreover, without loss of generality let $\rx{\gamma}\colon [0,1] \to \rx{\widehat{D}}$ be an $\rx{\widehat{R}}$--connected curve that agrees with $\gamma$ on $\widehat{D}$, i.e.\ each section of $\rx{\gamma}$ on $\widehat{D}$ is identical, up to time scaling, to a section of $\gamma$.
  Thus $\rx{\gamma}$ has at most $k$ $\rxsymbol$--length extra sections, $L(\gamma) \leq L(\rx{\gamma}) \leq L(\gamma) + k \rxsymbol$, and $\indmetric{\rx{\cal M}}(p,q) \leq L(\rx{\gamma}) \leq \indmetric{\cal M}(p,q) + k \rxsymbol + \delta$.
  But this inequality is valid for each $\delta > 0$, hence $\indmetric{\rx{\cal M}}(p,q) \leq \indmetric{\cal M}(p,q) + k\rxsymbol$.
  The result follows after taking the limit as $\rxsymbol \to 0$.
  % The result follows after noting this inequality is valid for all $\delta > 0$, thus $\indmetric{\rx{\cal M}}(p,q) \leq \indmetric{\cal M}(p,q) + k\rxsymbol$.
\end{proof}
\noindent Note that Theorem~\ref{thm:metric_convergence} does not imply that the topology of $\rx{\cal M}$ converges to the topology of ${\cal M}$.
On the contrary, $\rx{\cal M}$ is homotopically equivalent to the graph $({\cal J},\Gamma)$ for each $\rxsymbol > 0$~\cite{Ames2005b}, whereas the topology of ${\cal M}$ may be different~\cite{Simic2005}.

We conclude this section by introducing metrics between curves on $\rx{\cal M}$.
\begin{definition}
  \label{def:curve_metric}
  Let $I \subset [0,\infty)$ a bounded interval.
  Given any two curves $\gamma, \gamma'\colon I \to \rx{\cal M}$, we define:
  \begin{equation}
    \rx{\rho_I} \parenb{\gamma, \gamma'} = \sup\setb{\indmetric{\rx{\cal M}}\parenb{\gamma(t), \gamma'(t)} \mid t \in I}.
  \end{equation}
\end{definition}
\noindent Our choice of the supremum among point--wise distances in Definition~\ref{def:curve_metric} is inspired by the $\sup$--norm for continuous real--valued functions.

\section{Relaxed Executions and Discrete Approximations}
\label{sec:simulation}

This section contains our main result: discrete approximations of executions of controlled hybrid systems, constructed using any variable step size numerical integration algorithm, converge uniformly to the actual executions.
This section is divided into three parts.
First, we define a pair of algorithms that construct executions of controlled hybrid systems and their relaxations, respectively.
Next, we develop a discrete approximation scheme for executions of relaxations of controlled hybrid systems.
Finally, we prove that these discrete approximations converge to orbitally stable executions of the original, non--relaxed, controlled hybrid system using the metric topologies developed in Section~\ref{sec:relaxation}.

%%
%% Definition of an execution
%%
\subsection{Execution of a Hybrid System}
\label{subsec:hybrid_execution}

We begin by defining an \emph{execution} of a controlled hybrid system.
This definition agrees with the traditional intuition about executions of controlled hybrid systems which describes an execution as evolving as a standard control system until a guard is reached, at which point a discrete transition occurs to a new domain using a reset map.
We provide an explicit definition to clarify technical details required in the proofs below.
Given a controlled hybrid system, ${\cal H}$, as in Definition~\ref{def:hybrid_system}, the algorithm in Fig.~\ref{fig:algo_hybrid_execution} defines an execution of ${\cal H}$ via construction.
A resulting execution, denoted $x$, is an $\widehat{R}$--connected curve from some interval $I \subset [0,\infty)$ to $\coprod_{j \in {\cal J}} D_j$.
Thus, abusing notation, we regard $x$ as a continuous curve on ${\cal M}$.
Abusing notation again, we regard $x$ as a piece--wise continuous curve on $\rx{\cal M}$ for each $\rxsymbol > 0$.
Fig.~\ref{fig:hybrid_execution} shows an execution undergoing a discrete transition.

%% Using 'figure' instead of 'algorithm', as explained in Section X.B of IEEEtran_HOWTO.pdf
\begin{figure}[tp]
  \begin{algorithmic}[1]
    \Require $t = 0$, $j \in {\cal J}$, $p \in D_j$, and $u \in BV(\R,U)$.
    \State Set $x(0) = p$.
    \Loop
    \State Let $\gamma\colon I \to D_j$ be the maximal integral curve of $f_j$ with control $u$ such that $\gamma(t) = x(t)$.\label{algo:hexec_max_curve}
    \State Let $t' = \sup I$ and $x(s) = \gamma(s)\ \forall s \in [t,t')$.
    \Statex \Comment \emph{Note if $t' < \infty$, then $\gamma(t') \in \bdry{D_j}$}.
    \If{$t' = \infty$, \textbf{or} $\nexists e \in {\cal N}_j$ such that $\gamma(t') \in G_e$}\label{algo:hexec_step_4}
    \State Stop.
    \EndIf
    \State Let $(j,j') \in {\cal N}_j$ be such that $\gamma(t') \in G_{(j,j')}$.\label{algo:hexec_choice_jump}
    % \Comment Now we perform a mode transition.
    \State Set $x(t') = R_{(j,j')}\parenb{\gamma(t')}$, $t = t'$, and $j = j'$.\label{algo:hexec_reset}
    \Statex \Comment \emph{Note $\gamma(t') \simfunc{\widehat{R}} x(t')$}.
    \EndLoop
  \end{algorithmic}
  \caption{Algorithm to construct an execution of a controlled hybrid system ${\cal H}$.}
  \label{fig:algo_hybrid_execution}
\end{figure}

%%
%% Uniqueness of executions
%%
Note that executions constructed using the algorithm in Fig.~\ref{fig:algo_hybrid_execution} are not necessarily unique.
Indeed, Definition~\ref{def:maximal_curve} implies that once a discrete transition has been performed, the execution is unique until a new transition is made; however, the choice in Step~\ref{algo:hexec_choice_jump} is not necessarily unique if the maximal integral curve passes through the intersection of multiple guards.
It is not hard to prove that a sufficient condition for uniqueness of executions is that all the guards are disjoint, even though, as we show in Section~\ref{sec:loc}, uniqueness of the executions can be obtained for some cases where guards do intersect.
% We are interested in systems that produce unique executions, thus we impose the following assumption.
% \begin{assumption}
%   \label{assump:unique_hybrid_execution}
%   Given a hybrid system ${\cal H}$, every execution is unique.
% \end{assumption}

\begin{figure*}[tp]
  \mbox{}%
  \hfill%
  \subfloat%
  [Discrete transition of an execution $x$.\label{fig:hybrid_execution}]%
  {%
    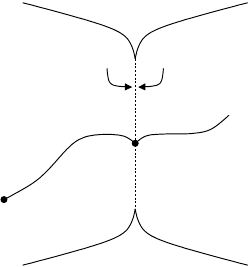%
  }
  \hfill
  \subfloat%
  [Zeno execution $x$ accumulating at $p'$.\label{fig:zeno_accumulation}]%
  {%
    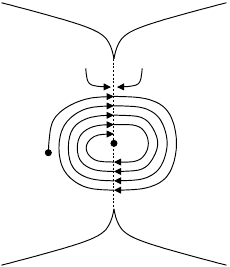%
  }
  \hfill
  \subfloat%
  [Non--orbitally stable execution at initial condition $p'$.\label{fig:orbital_stability}]%
  {%
    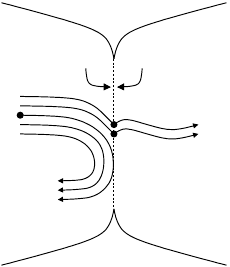%
  }
  \hfill%
  \mbox{}%
  \caption{Examples of different executions for a two--mode hybrid system.}
\end{figure*}

%%
%% Definition of Zeno executions
%%
With the definition of execution of a controlled hybrid system, we can define a class of executions unique to controlled hybrid systems.
\begin{definition}
  \label{def:zeno_execution}
  An execution is \emph{Zeno} when it undergoes an infinite number of discrete transitions in a finite amount of time.
  Hence, there exists $T > 0$, called the \emph{Zeno time}, such that the execution is only defined on $I = [0,T)$.
\end{definition}
\noindent Zeno executions are hard to simulate since they apparently require an infinite number of reset map evaluations, an impossible task to implement on a digital computer.
A consequence of the algorithm in Fig.~\ref{fig:algo_hybrid_execution} is that if $x\colon I \to {\cal M}$ is an execution such that $T = \sup I < \infty$, then either
\begin{enumparen}
\item $x$ has a finite number of discrete transitions on $I = [0,T]$, and $x(T) \in \bdry{D_j}$ for some $j \in {\cal J}$, or
\item $x$ is a Zeno execution and $I = [0,T)$.
\end{enumparen}

%%
%% Zeno accumulation
%%
We now introduce a property of Zeno executions of particular interest in this paper:
\begin{definition}
  \label{def:zeno_accumulation}
  Let ${\cal H}$ be a controlled hybrid system, $p \in {\cal M}$, $u \in BV(\R,U)$, and $x\colon [0,T) \to {\cal M}$ be a Zeno execution with initial condition $p$, control $u$, and Zeno time $T$.
  % We say that $x$ \emph{accumulates on $p' \in {\cal M}$} if for each $r > 0$ there exists $\delta > 0$ such that:
  % \begin{equation}
  %   \sup_{t \in [T-\delta, T)} \indmetric{\cal M}\parenb{p',x(t)} \leq r.
  % \end{equation}
  $x$ \emph{accumulates at $p' \in {\cal M}$} if $\lim_{t \to T} \indmetric{\cal M}\parenb{x(t),p'} = 0$.
\end{definition}
% If a Zeno execution accumulates then in practice it is impossible for the execution to ``fill'' a portion of the hybrid space while having an infinite number of discrete transitions.
\noindent Examples of Zeno executions that do not accumulate can be found in~\cite{Zhang2001}.
Fig.~\ref{fig:zeno_accumulation} shows a Zeno execution that accumulates at $p'$.
Note that for $p'$ to be a Zeno accumulation point, it must belong to a guard of a controlled hybrid system.

%%
%% Definition of orbital stability
%%
Since ${\cal M}$ is a metric space, we can introduce the concept of continuity of a hybrid execution with respect to its initial condition and control input in a straightforward way. Employing this definition, we can define the class of executions that are numerically approximable:
\begin{definition}
  \label{def:orbital_stability}
  Let ${\cal H}$ be a controlled hybrid system.
  Denote by $x_{(p,u)}\colon I_{(p,u)} \to {\cal M}$ a hybrid execution of ${\cal H}$ with initial condition $p \in {\cal M}$ and control $u \in BV(\R,U)$.
  Given $T > 0$, we say that $x_{(p,u)}$ is \emph{orbitally stable in $[0,T]$ at $(p,u)$} if there exists a neighborhood of $(p,u)$, say $N_{(p,u)} \subset {\cal M} \times BV(\R,U)$, such that:
  \begin{enumparen}
  \item $x_{(p',u')}$ is unique for each $(p',u') \in N_{(p,u)}$.
  \item $[0,T] \subset I_{(p',u')}$ for each $(p',u') \in N_{(p,u)}$.
  \item The map $(p',u') \mapsto x_{(p',u')}(t)$ is continuous at $(p,u)$ for each $t \in [0,T]$.
  \end{enumparen}
  % Given $T > 0$, we say that the map $(p,u) \mapsto x_{(p,u)}$ is \emph{orbitally stable in $[0,T]$} at $(p',u') \in {\cal M} \times BV(\R,U)$ if there exists a neighborhood of $(p',u')$, say $N_{(p',u')} \subset {\cal M} \times BV(\R,U)$, such that the following conditions are satisfied:
  % \begin{enumparen}
  % \item $[0,T] \subset I_{(p,u)}$ for each $(p,u) \in N_{(p',u')}$.
  % \item The map $(p,u) \mapsto x_{(p,u)}(t)$ is continuous at $(p',u')$ for each $t \in [0,T]$.
  % \end{enumparen}
\end{definition}
% sam 1.20, 3.1
%\begin{remark}
\noindent
As observed in Section III.B in~\cite{Lygeros2003}, executions that are not orbitally stable are difficult to approximate with a general algorithm.
Figure~\ref{fig:orbital_stability} shows a non--orbitally stable execution that intersects the guard tangentially, and note that executions initialized arbitrarily close to $p' \in D_1$ undergo different sequences of transitions.
Unfortunately, there is presently no general test (analytical or numerical) that ensures a given execution is orbitally stable.
% However, there is an easily--verified sufficient condition that covers the range of examples considered in this paper.
Theorem III.2 in~\cite{Lygeros2003} provides one set of sufficient conditions ensuring orbital stability.
%\end{remark}

%% sam 1.20, 3.1
%\begin{proposition}\label{prop:orb}
%Let $x:[0,T]\into{\cal M}$ be an execution in a controlled hybrid system ${\cal H} = ({\cal J}, \Gamma, {\cal D}, U, {\cal F}, {\cal G}, {\cal R})$ over the finite time interval $[0,T]$.
%Let $\set{t_i}_{i=0}^k\subset [0,T]$ denote the sequence of discrete transition times, so that $t_0 = 0$ and $t_k = T$.
%If $\set{t_i}_{i=0}^k$ is strictly increasing and:
%\begin{enumerate}
%  \item there exists $\set{j_i}_{i=0}^k$ such that for all $i\in\set{0,\dots,k-1}$ we have $x|_{[t_i,t_{i+1})} \subset \intoper D_{j_i}$;
%  \item for all $i\in\set{1,\dots,k}$ there exists a unique $(j,j')\in\Gamma$ such that with $x_i^+ = \lim_{t\goesto t_i} x|_{[t_i,t_{i+1})}$ we have $x_i^+ \in \intoper G_{(j,j')}$ and $F_j$ is transverse to $G_{(j,j')}$ at $x_i^+$;
%\end{enumerate}
%then $x$ is orbitally stable in $[0,T]$.
%\end{proposition}
%
%\noindent
%This result is a straightforward consequence of Theorem III.2 in~\cite{Lygeros2003}.

%\sam{Where in that paper do you see this property referred to?  Definition II.5?  We need a better reference for this idea, or we need to add a whole subsection explaining it.}
%Indeed, if an execution is not orbitally stable then there exists a time $t'$ such that, if another execution is initialized arbitrarily close to $x(t')$ or if an arbitrarily close control is applied, then the pair of executions travel through different sequences of discrete transitions and can never be made arbitrarily close.

%%
%% Definition of relaxed executions
%%
\subsection{Relaxed Execution of a Hybrid System}
\label{sec:rx_execution}

Next, we define the concept of \emph{relaxed execution} for a relaxation of a controlled hybrid system.
The main idea is that, once a relaxed execution reaches a guard, we continue integrating over the strip with the relaxed vector field, $f_e$, as in Definition~\ref{def:relaxed_vf}.
Given the controlled hybrid system, ${\cal H}$ and its relaxation, $\rx{\cal H}$ for some $\rxsymbol > 0$, the algorithm in Fig.~\ref{fig:algo_relaxed_execution} defines a relaxed execution of $\rx{\cal H}$ via construction.
The resulting relaxed execution, denoted $\rx{x}$, is a continuous function defined from an interval $I \subset [0,\infty)$ to $\rx{\cal M}$.
Note that this algorithm is only defined for initial conditions belonging to $D_j$ for some $j \in {\cal J}$ since the strips are artificial objects that do not appear in ${\cal H}$.
The generalization to all initial conditions is straighforward; we omit it to simplify the presentation.

%% Using 'figure' instead of 'algorithm', as explained in Section X.B of IEEEtran_HOWTO.pdf
\begin{figure}[tp]
  \begin{algorithmic}[1]
    \Require $t = 0$, $j \in {\cal J}$, $p \in D_j$, and $u \in BV(\R,U)$.
    \State Set $\rx{x}(0) = p$.
    \Loop
    \State Let $\gamma\colon I \to D_j$, the maximal integral curve of $f_j$ with control $u$ such that $\gamma(t) = \rx{x}(t)$.\label{algo:rexec_max_curve}
    \State Let $t' = \sup I$ and $\rx{x}(s) = \gamma(s)\ \forall s \in [t,t')$.
    \Statex \Comment \emph{Note if $t' < \infty$, then $\gamma(t') \in \bdry{D_j}$}.
    \If{$t' = \infty$, \textbf{or} $\nexists e \in {\cal N}_j$ such that $\gamma(t') \in G_e$}\label{algo:rexec_step_4}
    \State Stop.
    \EndIf
    \State Let $(j,j') \in {\cal N}_j$ such that $\gamma(t') \in G_{(j,j')}$, hence $\parenb{\gamma(t'),0} \in \rx{S_{(j,j')}}$.\label{algo:rexec_choice_jump}
    \State Set $\rx{x}(t' + \tau) = \parenb{\gamma(t'),\tau}\ \forall \tau \in [0,\rxsymbol)$.\label{algo:rexec_strip_integral}
    \State Set $\rx{x}(t' + \rxsymbol) = \rx{R_{(j,j')}}\parenb{\gamma(t'),\rxsymbol}$, $t = t' + \rxsymbol$, and $j = j'$.
    \Statex \Comment \emph{Note $\parenb{\gamma(t'),\rxsymbol}  \simfunc{\rx{\widehat{R}}}  \rx{x}(t' + \rxsymbol)$}.
    \EndLoop
  \end{algorithmic}
  \caption{Algorithm to construct a relaxed execution of a relaxation of a controlled hybrid system, $\rx{\cal H}$.}
  \label{fig:algo_relaxed_execution}
\end{figure}

Step~\ref{algo:rexec_strip_integral} of the algorithm in Fig.~\ref{fig:algo_relaxed_execution} relaxes each instantaneous discrete transition by integrating over the vector field on a strip, hence forming a continuous curve on $\rx{\cal M}$.
Also note that our definition for the relaxed execution over each strip $\rx{S_e}$, also in Step~\ref{algo:rexec_strip_integral}, is exactly equal to the maximal integral curve of $f_e$.
Fig.~\ref{fig:relaxed_execution} shows an example of a relaxed mode transition produced by the algorithm in Fig.~\ref{fig:algo_relaxed_execution}.
Given a hybrid system ${\cal H}$ and its relaxation $\rx{\cal H}$, the relaxed execution of $\rx{\cal H}$ produced by the algorithm in Fig.~\ref{fig:algo_relaxed_execution} is a delayed version of the execution of ${\cal H}$ produced by the algorithm in Fig.~\ref{fig:algo_hybrid_execution}, since the relaxed version has to expend $\rxsymbol$ time units during each discrete transition.
In that sense, our definition of relaxed execution is equivalent to an execution of a \emph{regularized hybrid system}~\cite{JohanssonEgerstedt1999}.

\begin{figure}[t]
  \centering
  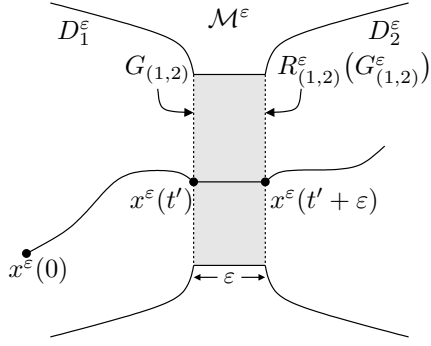%
  \caption{Relaxed mode transition of a relaxed execution $\rx{x}$ in a two--mode relaxed hybrid dynamical system.}%
  \label{fig:relaxed_execution}%
\end{figure}

Note that if a relaxed execution is unique for a given initial condition and input, then the corresponding hybrid execution is also unique, but not vice versa.
Indeed, consider the case of a hybrid execution performing a single discrete transition at a point, say $p$, where two guards intersect, i.e.\ $p \in G_e$ and $p \in G_{e'}$, such that $R_e(p) = R_{e'}(p)$.
In this case the hybrid execution is unique, but its relaxed counterpart either evolves via $S_e$ or $S_{e'}$, hence obtaining $2$ different executions.
Nevertheless, both relaxed executions reach the same point after evolving over the strip.

%%
%% Theorem: relaxed executions converge to hybrid executions
%%
Next, we state our first convergence theorem.
\begin{theorem}
  \label{thm:relaxed_execution_convergence}
  Let ${\cal H}$ be a controlled hybrid system and $\rx{\cal H}$ be its relaxation.
  Let $p \in {\cal M}$, $u \in BV(\R,U)$, $x\colon I \to \rx{\cal M}$ be an execution of ${\cal H}$ with initial condition $p$ and control $u$, and let $\rx{x}\colon \rx{I} \to \rx{\cal M}$ be a corresponding relaxed execution of $x$.
  Assume that the following conditions are satisfied:
  \begin{enumparen}
  \item $x$ is orbitally stable with initial condition $p$ and control $u$;
  \item $x$ has a finite number of discrete transitions or is a Zeno execution that accumulates; and
  \item there exists $T > 0$ such that for each $\rxsymbol$ small enough, $[0,T] \subset I \cap \rx{I}$ if $x$ has a finite number of discrete transitions, and $[0,T)\subset I \cap \rx{I}$ if $x$ is Zeno.
  \end{enumparen}
  Then, $\lim_{\rxsymbol \to 0} \rx{\rho_{[0,T]}}\parenb{x,\rx{x}} = 0$.
\end{theorem}
\begin{proof}
  We provide the main arguments of the proof, omitting some details in the interest of brevity.
  First, given $j \in {\cal J}$ and $[\tau,\tau') \subset [0,T]$ such that $x(t) \in D_j$ for each $t \in [\tau,\tau')$, then, since $x|_{[\tau,\tau')}$ is absolutely continuous, for each $t,t' \in [\tau,\tau')$,
  \begin{equation}
    \label{eq:proof_thm27_xlipschitz}
    \indmetric{\rx{\cal M}}\parenb{x(t),x(t')}
    \leq L_{\indmetric{D_j}}\parenb{x|_{[t,t')}}
    = \int_t^{t'} \norm{f_j\parenb{s,x(s),u(s)}} \diff{s}
    \leq K (t' - t),
  \end{equation}
  where $K = \sup\setb{\normb{f_j\parenb{t,x,u}} \mid j \in {\cal J},\ t \in [0,T],\ x \in \rx{\cal M}, u \in U} < \infty$.

  Second, let $k \in \N$ and $\set{\lambda_i}_{i=0}^k \subset [0,1]$ be a sequence such that $0 = \lambda_0 \leq \lambda_1 \leq \ldots \leq \lambda_k = 1$.
  Given $\rxsymbol > 0$, let $\gamma_t\colon [0,1] \to \rx{\cal M}$ be defined by $\gamma_t(\lambda) = \rxcustom{x}{\lambda\rxsymbol}(t)$.
  Thus, by Theorem~\ref{thm:metric_convergence} and the algorithm in Figure~\ref{fig:algo_relaxed_execution}, $\gamma_t(0) = \rxcustom{x}{0}(t) = x(t)$ and $\gamma_t(1) = \rx{x}(t)$.
  Assume that $\rx{x}(t) \in D_j$ for each $t \in [\tau+\rxsymbol,\tau'+\rxsymbol)$, where $[\tau,\tau')$ is as defined above.
  Using Picard's Lemma (Lemma~5.6.3 in~\cite{Polak1997}), for each $t \in [\tau+\rxsymbol,\tau')$,
  \begin{multline}
    \normb{\rx{x}(t+\rxsymbol) - x(t)} \leq e^{L(t-\tau)} \biggl(\norm{\rx{x}(\tau+\rxsymbol) - x(\tau)} +\\
    + \int_\tau^t \norm{f_j\parenb{s,x(s),u(s)} - f_j\parenb{s+\rxsymbol,x(s),u(s+\rxsymbol)}} \diff{s}\biggr)\\
    \leq e^{L(t-\tau)} \biggl(\norm{\rx{x}(\tau+\rxsymbol) - x(\tau)}
    + L \int_\tau^t \rxsymbol + \norm{u(s) - u(s+\rxsymbol)} \diff{s}\biggr)\\
    \leq e^{L(t-\tau)}  \parenB{\norm{\rx{x}(\tau+\rxsymbol)  -  x(\tau)}  + 
      \parenb{L  +  \totvar{u}} (t-\tau) \rxsymbol},
  \end{multline}
  where we have used a standard property of the functions of bounded variation (Exercise~5.1 in~\cite{Ziemer1989}).
  Thus, if we assume that $\norm{\rx{x}(\tau+\rxsymbol) - x(\tau)} = \bigo{\rxsymbol}$, i.e.\ that there exists $C > 0$ such that $\norm{\rx{x}(\tau+\rxsymbol) - x(\tau)} \leq C \rxsymbol$, then $\norm{\rx{x}(t+\rxsymbol) - x(t)} = \bigo{\rxsymbol}$ for each $t \in [\tau+\rxsymbol,\tau')$.
  Using the same argument as above
  $\norm{\rxcustom{x}{\lambda_{i+1}\rxsymbol}(t+\rxsymbol) - \rxcustom{x}{\lambda_i \rxsymbol}(t)}
  = \bigo{(\lambda_{i+1} - \lambda_i) \rxsymbol}$,
  which implies that $\gamma_t$ is continuous for each $t \in [\tau+\rxsymbol,\tau')$, and that $L(\gamma_t) = \bigo{\rxsymbol}$, hence $\indmetric{D_j}\parenb{\rx{x}(t+\rxsymbol),x(t)} = \bigo{\rxsymbol}$.

  Assuming now that $x$ performs $2$ discrete transitions at times $\tau,\tau' \in [0,T]$, such that $\tau + \rxsymbol < \tau'$, transitioning from mode $j$ to $j'$, and the from mode $j'$ to $j''$.
  Note that, by definition, $x|_{[0,\tau)} = \rx{x}|_{[0,\tau)}$.
  Moreover, since $x$ is orbitally stable, we know that $\rx{x}$ performs the same $2$ discrete transitions for $\rxsymbol$ small enough.
  Let $\rx{\tau}+\rxsymbol \in [0,T]$ be such that $\rx{x}(\rx{\tau}+\rxsymbol) \in G_{(j',j'')}$.
  Note that $\abs{\rx{\tau} - \tau'} = \bigo{\rxsymbol}$ since $\rx{x} \to x$ uniformly and $x$ is Lipschitz continuous (both propositions shown above).
  Assume that $\tau' \leq \rx{\tau}+\rxsymbol$ and consider the following upper bounds:
  \begin{enumparen}
  \item If $t \in [\tau,\tau+\rxsymbol)$, then $x(t) \in D_{j'}$ and $\rx{x}(t) \in \rx{S_{(j,j')}}$, thus:
    \begin{equation}
      \label{eq:thm26_ineq1}
      \indmetric{\rx{\cal M}}\parenb{x(t),\rx{x}(t)}
      \leq \indmetric{D_{j'}}\parenb{x(t),x(\tau)}
      + d_\rx{S_{(j,j')}}\parenb{x(\tau),\rx{x}(t)}
      = \bigo{\rxsymbol}.
    \end{equation}
  \item \label{thm:27_proof_case2}
    If $t \in [\tau+\rxsymbol,\tau')$, then $x(t),\rx{x}(t) \in D_{j'}$, thus, using the bound obtained above:
    \begin{equation}
      \indmetric{\rx{\cal M}}\parenb{x(t),\rx{x}(t)}
      \leq \indmetric{D_{j'}}\parenb{x(t),x(t-\rxsymbol)}
      + \indmetric{D_{j'}}\parenb{x(t-\rxsymbol),\rx{x}(t)}
      = \bigo{\rxsymbol}.
    \end{equation}
  \item If $t \in [\tau',\rx{\tau}+\rxsymbol)$, then $x(t) \in D_{j''}$ and $\rx{x} \in D_{j'}$, thus, denoting $\lim_{t \uparrow \tau'} x(t) = x(\tau'_-)$:
    \begin{multline}
      \indmetric{\rx{\cal M}}\parenb{x(t),\rx{x}(t)}
      \leq \indmetric{D_{j''}}\parenb{x(t),x(\tau')}
      + d_\rx{S_{(j',j'')}}\parenb{x(\tau'),x(\tau'_-)} +\\
      + \indmetric{D_{j'}}\parenb{x(\tau'_-),\rx{x}(\rx{\tau}+\rxsymbol)}
      + \indmetric{D_{j'}}\parenb{\rx{x}(\rx{\tau}+\rxsymbol),\rx{x}(t)}
      \leq \bigo{\rxsymbol}.
    \end{multline}
  \item If $t \in [\rx{\tau}+\rxsymbol,\rx{\tau}+2\rxsymbol)$, then $x(t) \in D_{j''}$ and $\rx{x} \in \rx{S_{(j'.j'')}}$, thus:
    \begin{equation}
      \indmetric{\rx{\cal M}}\parenb{x(t),\rx{x}(t)}
      \leq \indmetric{D_{j''}}\parenb{x(t),x(\tau')}
      + d_\rx{S_{(j',j'')}}\parenb{x(\tau'),\rx{x}(t)}
      \leq \bigo{\rxsymbol}.
    \end{equation}
  \item If $t \in [\rx{\tau}+2\rxsymbol,T]$, then $x(t),\rx{x}(t) \in D_{j''}$, thus we get the same bound as in  case~\ref{thm:27_proof_case2}.
  \end{enumparen}
  Therefore, $\rx{\rho_{[0,T]}}(x,\rx{x}) = \bigo{\rxsymbol}$ as desired.
  Note that the general case, with an arbitrary number of discrete transitions, follows by using the a similar argument as above by properly considering the time intervals and then applying the upper bounds inductively.

  Next, let us consider the case when $x$ is a Zeno execution that accumulates on $p'$.
  Let $\delta > 0$, then $x|_{[0,T-\delta]}$ has a finite number of discrete transitions, and as shown above, $\indmetric{\rx{\cal M}}\parenb{x(T-\delta),\rx{x}(T-\delta)} = \bigo{\rxsymbol}$.
  Moreover, $\indmetric{\rx{\cal M}}\parenb{x(T-\delta),x(t)} = \bigo{\delta}$ and $\indmetric{\rx{\cal M}}\parenb{\rx{x}(T-\delta),\rx{x}(t)} = \bigo{\delta}$ for each $t \in [T-\delta,T)$.
  The conclusion follows by noting that these bounds are valid for each $\delta > 0$.
\end{proof}

%%
%% Definition of discrete approximation
%%
\subsection{Discrete Approximations}
Finally, we are able to define the \emph{discrete approximation of a relaxed execution}, which is constructed as an extension of any existing ODE numerical integration algorithm.
Given a controlled hybrid system ${\cal H}$, $\dalgo{{\cal A}_j}\colon \R \times \R^{n_j} \times U \to \R^{n_j}$, where $\stepsymbol > 0$ and $j \in {\cal J}$, is a \emph{numerical integrator of order $\omega$}, if given $p \in D_j$, $u \in BV(\R,U)$, $x$ the maximal integral curve of $f_j$ with initial condition $p$ and control $u$, $N = \floor{\frac{T}{\stepsymbol}}$, and a sequence $\set{z_k}_{k=0}^N$ with $z_0 = p$ and $z_{k+1} = \dalgo{{\cal A}_j}\parenb{k\stepsymbol,z_k,u(k\stepsymbol)}$, then $\sup \setb{\norm{x(k\stepsymbol) - z_k} \mid k \in \set{0,\ldots,N}} = \bigo{\stepsymbol^\omega}$.
This definition of numerical integrator is compatible with commonly used algorithms, including Forward and Backward Euler algorithms and the family of Runge--Kutta algorithms (Chapter~7 in~\cite{LeVeque2007}).
The algorithm in Fig.~\ref{fig:algo_discrete_approx} defines a discrete approximation of a relaxed execution of $\rx{\cal H}$.
The resulting discrete approximation, for a step size $\stepsymbol > 0$, denoted by $\dapprox{z}$, is a function from a closed interval $I \subset [0,\infty)$ to $\rx{\cal M}$.

\begin{figure}[tp]
  \begin{algorithmic}[1]
    \Require $\stepsymbol > 0$, $k = 0$, $j \in {\cal J}$, and $p \in D_j$.
    \State Set $t_0 = 0$ and $\dapprox{z}(0) = p$.
    \Loop
    \State Set $n'  =  \inf\setb{n  \in  \N  \mid  \dalgocustom{{\cal A}_j}{\stepsymbol 2^{-n}} \parenb{t_k, \dapprox{z}(t_k), u(t_k)}  \in  \rx{D_j}}$.\label{algo:dapprox_find_n}
    \If{$n' = \infty$} \label{algo:dapprox_stop_cond}
    \State \Return $\dapprox{z}|_{[0,t_k]}$.
    \EndIf
    \State Set $t_{k+1} = t_k + \stepsymbol 2^{-n'}$.\label{algo:dapprox_tkpo}
    \State Set $\dapprox{z}(t_{k+1}) = \dalgocustom{{\cal A}_j}{\stepsymbol 2^{-n'}} \parenb{t_k, \dapprox{z}(t_k), u(t_k)}$.
    \State Set $\dapprox{z}(t) = \frac{t_{k+1}-t}{t_{k+1} - t_k} \dapprox{z}(t_k) + \frac{t-t_k}{t_{k+1} - t_k} \dapprox{z}(t_{k+1})$ $\forall t \in [t_k, t_{k+1}]$.
    \If{$\exists (j,j') \in {\cal N}_j$ such that $\dapprox{z}(t_{k+1}) \in \rx{S_{(j,j')}}$}
    \State Set $(q,\tau) = \dapprox{z}(t_{k+1}) \in \rx{S_{(j,j')}}$.
    \State Set $t_{k+2} = t_{k+1} + \epsilon - \tau$.
    \State Set $\dapprox{z}(t)  =  (q, t - t_{k+1} + \tau)\ \forall t  \in  [t_{k+1},t_{k+2})$.\label{algo:dapprox_strip_integral}
    \State Set $\dapprox{z}(t_{k+2})  =  \rx{R_{(j,j')}}(q,\rxsymbol)$, $k  =  k+2$, and $j  =  j'$.
    \Statex \Comment \emph{Note $(q,\rxsymbol)  \simfunc{\rx{\widehat{R}}}  \dapprox{z}(t_{k+2})$}.
    \Else
    \State Set $k = k+1$.
    \EndIf
    \EndLoop
  \end{algorithmic}
  \caption{Discrete approximation of a relaxed execution of the relaxation of a controlled hybrid system $\rx{\cal H}$.}
  \label{fig:algo_discrete_approx}
\end{figure}

We now make several remarks about the algorithm in Fig.~\ref{fig:algo_discrete_approx}.
First, the condition in Step~\ref{algo:dapprox_stop_cond} can only be satisfied, i.e.\ the Algorithm only stops, if $\dapprox{z}(t_k) \in \bdry{D_j}$ and $f_j\parenb{t_k,\dapprox{z}(t_k),u(t_k)}$ is outward--pointing, since otherwise a smaller step--size would produce a valid point.
Second, the function $\dapprox{z}$ is continuous on $\rx{\cal M}$.
Third, and most importantly, similar to the algorithm in Fig.~\ref{fig:algo_relaxed_execution}, the curve assigned to $\dapprox{z}$ in Step~\ref{algo:dapprox_strip_integral} is exactly the maximal integral curve of $f_e$ while on the strip.
By relaxing the guards using strips, and then endowing the strips with a trivial vector field, we avoid having to find the exact point where the trajectory intersects a guard.
Our relaxation does introduce an error in the approximation, but as we show in Theorem~\ref{thm:discrete_approx_conv}, the error is of order $\rxsymbol$.
Fig.~\ref{fig:discrete_approx} shows a discrete approximation produced by the algorithm in Fig.~\ref{fig:algo_discrete_approx} as it performs a mode transition.

\begin{figure}[t]
  \centering
  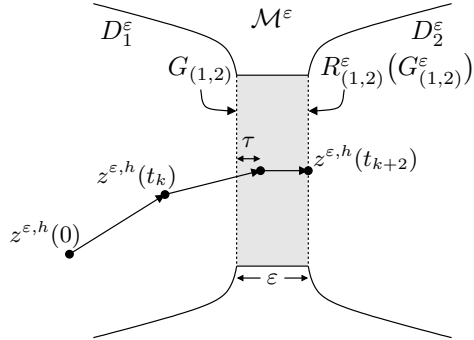%
  \caption{Discrete approximation $\dapprox{z}$ of a relaxed execution in a two-mode hybrid dynamical system.}%
  \label{fig:discrete_approx}
\end{figure}

\begin{theorem}
  \label{thm:discrete_approx_conv}
  Let ${\cal H}$ be a controlled hybrid system and $\rx{\cal H}$ its relaxation.
  Let $p \in {\cal M}$, $u \in BV(\R,U)$, and let $x\colon I \to \rx{\cal M}$ be a orbitally stable execution of ${\cal H}$ with initial condition $p$ and control $u$.
  Furthermore, let $\rx{x}\colon \rx{I} \to \rx{\cal M}$ be a relaxed execution with initial condition $p$ and control $u$, and let $\dapprox{z}\colon \dapprox{I} \to \rx{\cal M}$ be its discrete approximation.
  If $[0,T] \subset \rx{I} \cap \dapprox{I}$ for each $\rxsymbol$ and $\stepsymbol$ small enough, then there exists $C > 0$ such that $\lim_{h \to 0} \rx{\rho_{[0,T]}}\parenb{\rx{x},\dapprox{z}} \leq C \rxsymbol$.
\end{theorem}
\begin{proof}
  As we have done with the previous proofs, we only provide a sketch of the argument in the interest of brevity.
  Assume that $\rx{x}$ performs a single discrete transition in the interval $[0,T]$ for each $\rxsymbol$ small enough, crossing the guard $G_{(j,j')}$ at time $\rx{\tau}$.
  Let $\delta > 0$.
  Since $x$ is orbitally stable and $\dalgo{\cal A}$ is convergent with order $\omega$, for $\stepsymbol$ small enough there exists an initial condition $\dapprox{z}(0)$ such that $\absb{\rx{x}(0) - \dapprox{z}(0)} < \delta$ and $\dapprox{z}$ crosses the guard $G_{(j,j')}$ at time $\dapprox{\tau}_{k'} \in [t_{k'},t_{k'+1})$ for some $k' \in \N$, where $\set{t_k}_{k=0}^N$ is the set of time samples associated to $\dapprox{z}$.
  Moreover, we can choose $\stepsymbol$ small enough such that $\absb{\rx{\tau} - t_{k'+1}} \leq 2\delta + \bigo{\stepsymbol^\omega}$ and $\absb{t_{k'+2} - \rx{\tau}+\rxsymbol} = \bigo{\stepsymbol^\omega}$.

  Let $\sigma_m = \min\setb{t_{k'+1}, \rx{\tau}}$, $\sigma_M = \max\setb{t_{k'+1}, \rx{\tau}}$, $\nu_m = \min\setb{t_{k'+2}, \rx{\tau} + \rxsymbol}$, and $\nu_M = \max \setb{t_{k'+2}, \rx{\tau} + \rxsymbol}$.
  Also, let us assume that $h$ is small enough such that $\sigma_M \leq \nu_m$.
  Then on the interval $[0,\sigma_m)$ we get convergence due to $\dalgo{\cal A}$.
  On the interval $[\sigma_m,\sigma_M)$ one execution has transitioned into a strip, while the other is still governed by the vector field on $D_j$.
  On the interval $[\sigma_M,\omega_m)$ both executions are inside the strip, and on the interval $[\omega_m,\omega_M)$ one execution has transitioned to a new domain, while the second is still on the strip.
  After time $\omega_M$ both executions are in a new domain, and we can repeat the process.

  Consider the following cases:
  \begin{enumparen}
  \item By the convergence of algorithm $\dalgo{\cal A}$,
    \begin{equation}
      \indmetric{\rx{\cal M}}\parenb{\rx{x}(\sigma_m),\dapprox{z}(\sigma_m)} = \bigo{\delta} + \bigo{\stepsymbol^\omega}.
    \end{equation}
  \item \label{thm:proof_thm28_case2}
    Using~\eqref{eq:proof_thm27_xlipschitz} from the proof of Theorem~\ref{thm:relaxed_execution_convergence},
    \begin{multline}
      \indmetric{\rx{\cal M}}\parenb{\rx{x}(\sigma_M),\dapprox{z}(\sigma_M)} \leq
      \indmetric{\rx{\cal M}}\parenb{\rx{x}(\sigma_M),\rx{x}(\sigma_m)}
      + \indmetric{\rx{\cal M}}\parenb{\rx{x}(\sigma_m),\dapprox{z}(\sigma_m)} +\\
      + \indmetric{\rx{\cal M}}\parenb{\dapprox{z}(\sigma_m),\dapprox{z}(\sigma_M)}
      = \bigo{\delta} + \bigo{\stepsymbol^\omega}.
    \end{multline}
  \item Using the same argument as in the inequality above,
    \begin{equation}
      \hspace{-.9em}\indmetric{\rx{\cal M}}\parenb{\rx{x}(\nu_m),\dapprox{z}(\nu_m)}
      \leq
      \indmetric{\rx{\cal M}}\parenb{\rx{x}(\sigma_M),\dapprox{z}(\sigma_M)} + 2 \rxsymbol.
    \end{equation}
  \item Finally, again using the same argument as in case~\ref{thm:proof_thm28_case2},
    \begin{equation}
      \label{eq:thm27_ineq4}
      \indmetric{\rx{\cal M}}\parenb{\rx{x}(\nu_M),\dapprox{z}(\nu_M)}
      \leq
      \indmetric{\rx{\cal M}}\parenb{\rx{x}(\nu_m),\dapprox{z}(\nu_m)}
      + \bigo{\stepsymbol^\omega}.
    \end{equation}
  \end{enumparen}
  \noindent The generalization to any relaxed execution defined on $\rx{\cal M}$ and its discrete approximation follows by noting that they perform a finite number of discrete jumps on any bounded interval and that $\delta$ can be chosen arbitrarily small.
\end{proof}

Next, we state the main result of this Section, which is a result of Theorems~\ref{thm:relaxed_execution_convergence} and~\ref{thm:discrete_approx_conv}.
\begin{corollary}
  \label{cor:simulation}
  Let ${\cal H}$ be a hybrid dynamical system and $\rx{\cal H}$ be its relaxation.
  Let $p \in {\cal M}$, $u \in BV(\R,U)$, $x\colon I \to \rx{\cal M}$ be an execution of ${\cal H}$ with initial condition $p$ and control $u$, $\rx{x}\colon \rx{I} \to \rx{\cal M}$ be its corresponding relaxed execution, and $\dapprox{z}\colon \dapprox{I} \to \rx{\cal M}$ be its corresponding discrete approximation.
  If the following conditions are satisfied:
  \begin{enumparen}
    \item $x$ has a finite number of mode transitions or is a Zeno execution that accumulates;
    \item $x$ is orbitally stable; and,
    \item $[0,T] \subset I \cap \rx{I} \cap \dapprox{I}$ for each $\rxsymbol$ and $\stepsymbol$ small enough,
  \end{enumparen}
  then $\lim_{\substack{\rxsymbol \to 0\\ \stepsymbol \to 0}} \rx{\rho_{[0,T]}}\parenb{x,\dapprox{z}} = 0$.

  Moreover, the rate of convergence in the $\rx{\rho_{[0,T]}}$--metric is $\bigo{\epsilon} + \bigo{h^\omega}$.
\end{corollary}
\begin{proof}
  Note that, by Theorem~\ref{thm:relaxed_execution_convergence} together with the Triangle Inequality, this corollary is equivalent to proving that $\rx{\rho_I}\parenb{\rx{x},\dapprox{z}} \to 0$ as both $\rxsymbol,\stepsymbol \to 0$.
  Hence we show that $\rx{\rho_I}\parenb{\rx{x},\dapprox{z}}$ converges uniformly on $\stepsymbol$ as $\rxsymbol \to 0$.
  Using an argument similar to the one in the proof of Theorem~7.9 in~\cite{Rudin1964}, proving the uniform convergence on $h$ is equivalent to showing that
  $\lim_{\stepsymbol \to 0} \limsup_{\rxsymbol \to 0} \rx{\rho_I}\parenb{\rx{x},\dapprox{z}} = 0$,
  but this is true by Theorem~\ref{thm:discrete_approx_conv}, as desired.

  The rate of convergence follows from the proofs of Theorems~\ref{thm:relaxed_execution_convergence} and~\ref{thm:discrete_approx_conv}, in particular from inequalities~\eqref{eq:thm26_ineq1} to~\eqref{eq:thm27_ineq4}.
\end{proof}

\section{Examples}
\label{sec:examples}

We apply our results in three illustrative examples:
first
detailing the technical advantages of our intrinsic state--space metric over trajectory--space metrics in Section~\ref{sec:metric};
subsequently
comparing the performance of our provably--convergent simulation algorithm to the state--of--the--art in Section~\ref{sec:flo};
% subsequently implementing a benchmark hybrid system verification example in Section~\ref{sec:navbench};
and finally
applying our metric and simulation algorithm to a novel legged locomotion model in Section~\ref{sec:loc}.
Each example produces executions that are orbitally stable with respect to our state--space metric; this follows from~\cite[Theorem~2.8.3]{Filippov1988} for the examples in Sections~\ref{sec:metric} and~\ref{sec:loc} and~\cite[Theorem~5.1]{Schatzman1998} for the example in Section~\ref{sec:flo}.

\subsection{Metrization Example: Digital Control System}
\label{sec:metric}

We now study the distance between executions in the digital control system of Fig.~\ref{fig:motivation} using existing trajectory--space metrics and our proposed state--space metric.
Consider a nominal execution $x\colon [0,T] \into D$ that crosses the two thresholds simultaneously.
For each $\delta > 0$ let $y_\delta\colon [0,T] \into D$ be the execution initialized at
$y_\delta(0) = x(0) + (-\delta, 0)$
and let $z_\delta\colon [0,T] \into D$ be the execution initialized at
$z_\delta(0) = x(0) + (0, -\delta)$;
see Fig.~\ref{fig:dcs} for an illustration.
For each $\delta > 0$ the executions $y_\delta$ and $z_\delta$ undergo different sequences of logical controller states, $0 \goesto 1 \goesto 3$ or $0 \goesto 2 \goesto 3$, corresponding to transitions through different discrete modes in the controlled hybrid system in Fig.~\ref{fig:chs}.
In existing trajectory--space metrics~\cite{Tavernini1987, Gokhman2008, Tavernini2009, SanfeliceTeel2010}, $y_\delta$ and $z_\delta$ would be separated by at least unit distance.
In the state--space metric we develop in Section~\ref{sec:relaxation}, the distance between $y_\delta$ and $z_\delta$ in the controlled hybrid system of Fig.~\ref{fig:chs} is equal to that between the trajectories of the discontinuous vector field in Fig.~\ref{fig:dcs}, and in particular converges to zero as $\delta\goesto 0$.

An important consequence of this discussion is that $x$ is orbitally stable with respect to our state--space metric, but not with respect to existing trajectory--space metrics.
Therefore the algorithm described in Section~\ref{sec:simulation} is at present the only algorithm that yields simulations that provably converge to $x$.

\subsection{Simulation Example: Forced Linear Oscillator with Stop}
\label{sec:flo}

\begin{figure*}[tp]
  \centering
  \begin{minipage}{\textwidth}
    \subfloat%
    [Forced linear oscillator with stop.\label{fig:flo}]%
    {%
      \resizebox{.32\textwidth}{!}{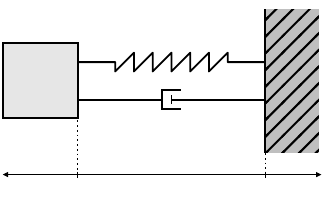}%
    }%
    \hfill%
    \subfloat%
    [Position of the analytical solution of Example~1 in Table~\ref{tab:msparams} (solid line), and position of the stop (dotted line).\label{fig:gtex1}]%
    {%
      \includegraphics[width=.32\textwidth,keepaspectratio=true]{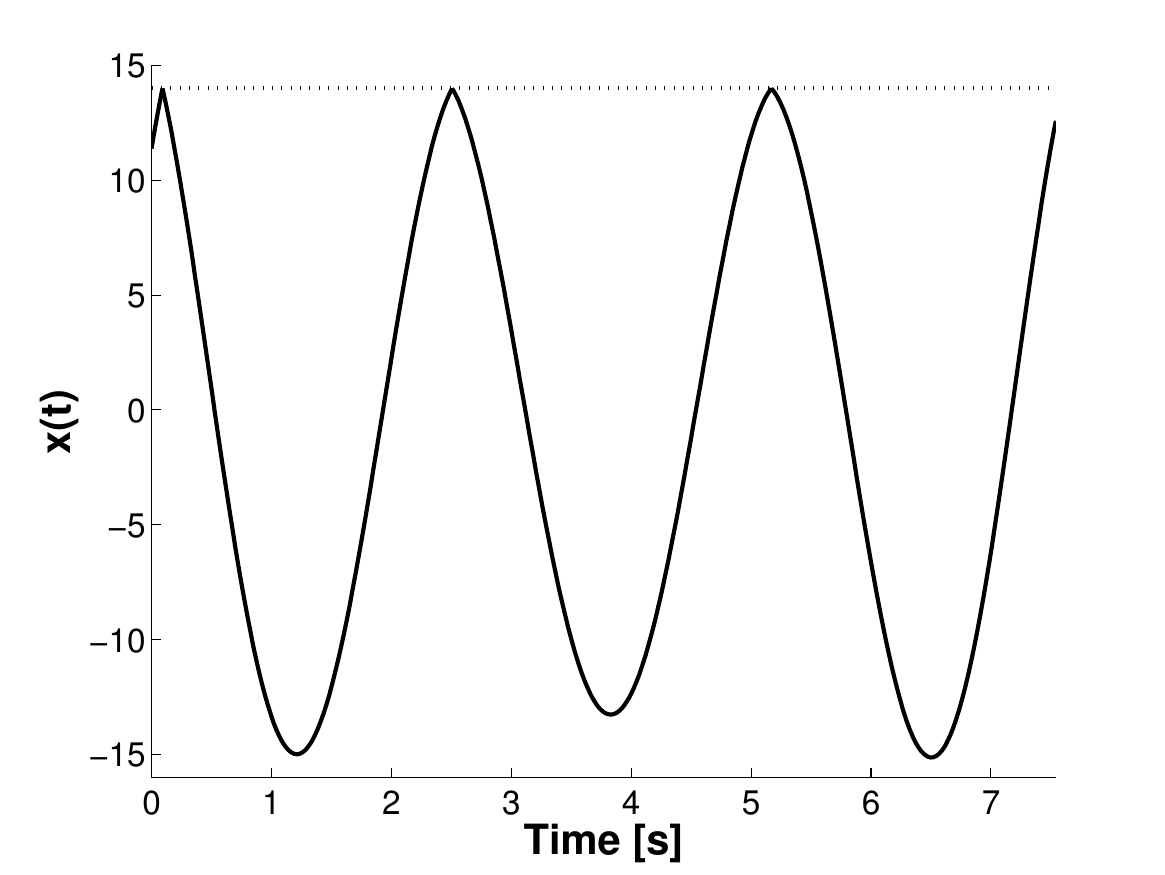}%
    }%
    \hfill%
    \subfloat%
    [Position of the analytical solution of Example~2 in Table~\ref{tab:msparams} (solid line), and position of the stop (dotted line).\label{fig:gtex2}]%
    {%
      \includegraphics[width=.32\textwidth,keepaspectratio=true]{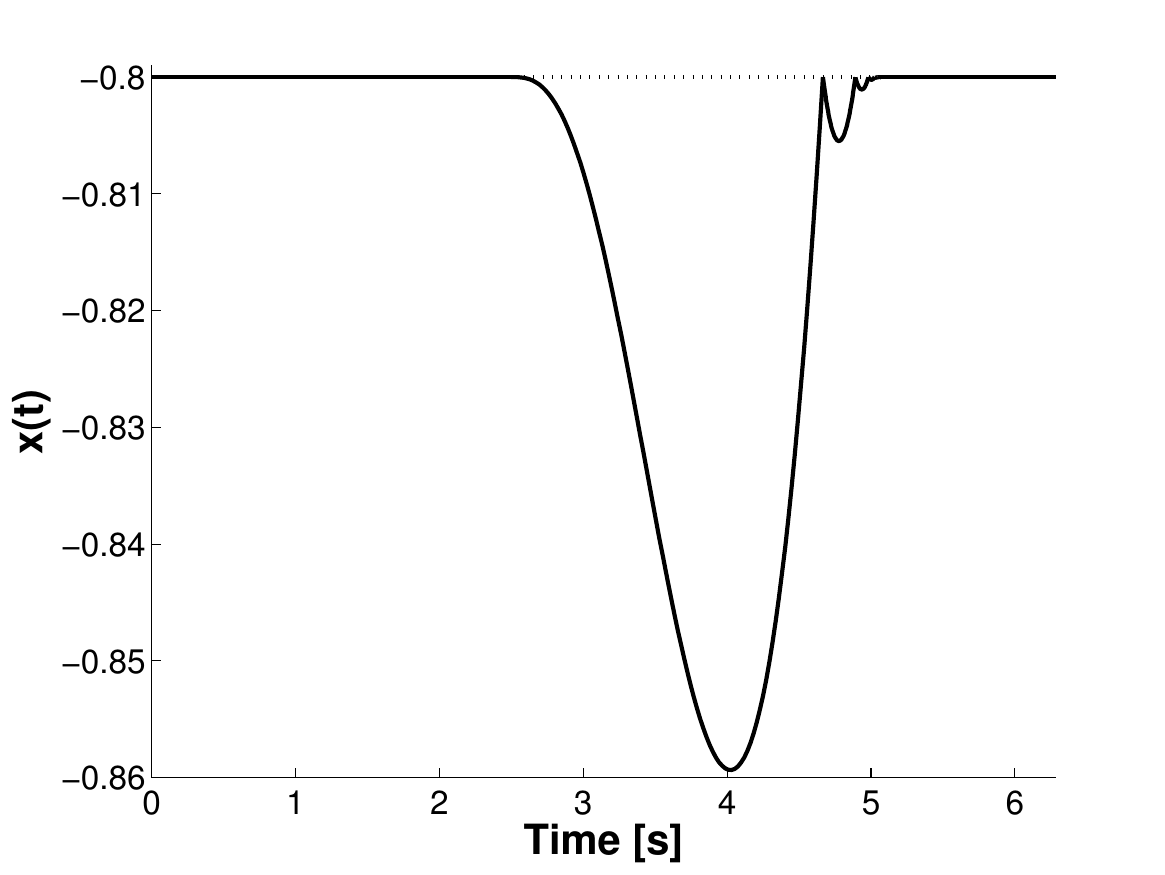}%
    }\\
    \subfloat%
    [$\rx{\rho_{[0,t_{\text{max}}]}}$--error of the algorithm in Fig.~\ref{fig:algo_discrete_approx} for the examples in Table~\ref{tab:msparams}.\label{fig:florho}]%
    {%
      \includegraphics[width=.32\textwidth,keepaspectratio=true]{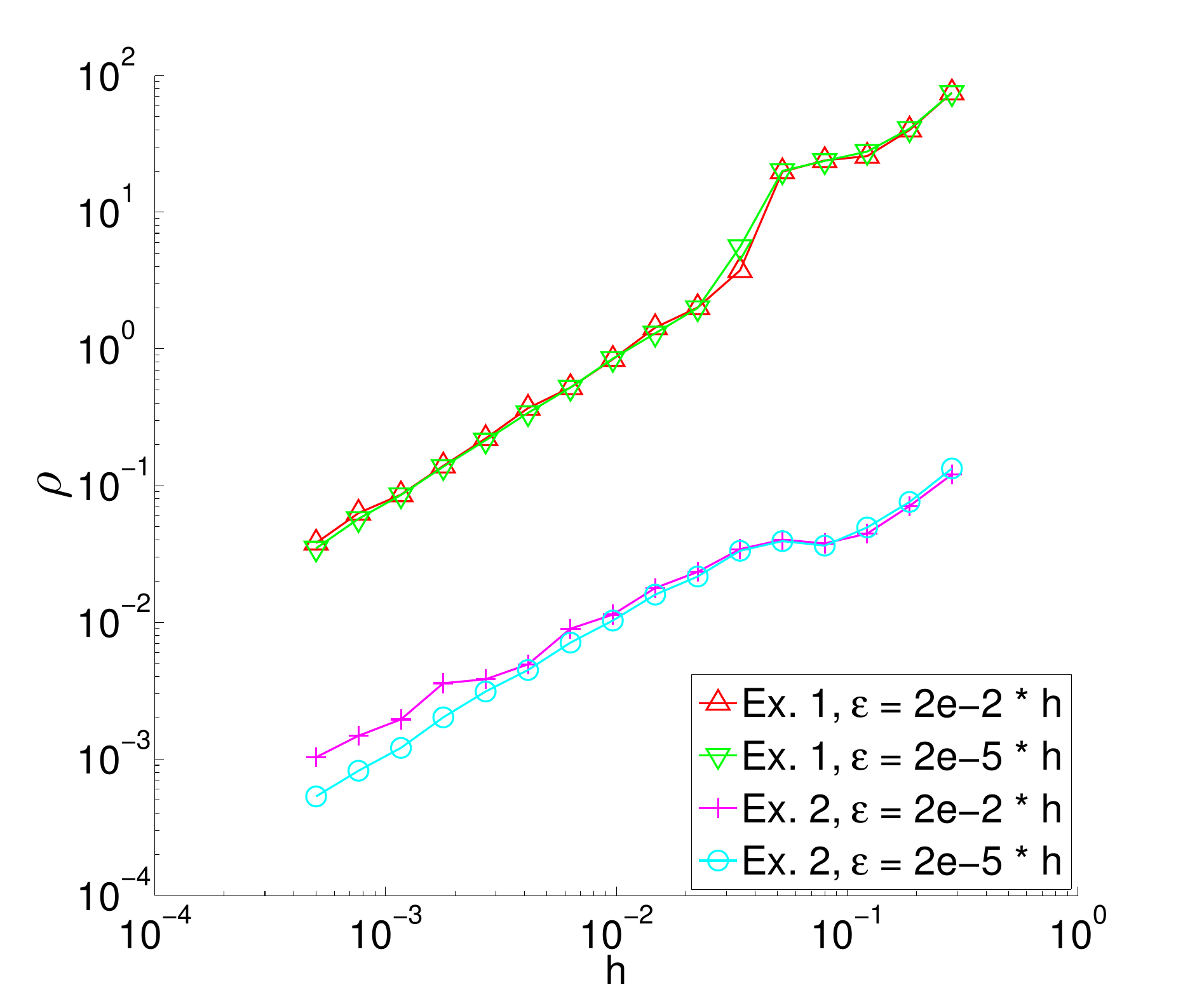}%
    }%
    \hfill%
    \subfloat%
    [$\hat{\rho}$--error of the algorithm in Fig.~\ref{fig:algo_discrete_approx} vs.\ the PS Method for the examples in Table~\ref{tab:msparams}.\label{fig:flomax}]%
    {%
      \includegraphics[width=.32\textwidth,keepaspectratio=true]{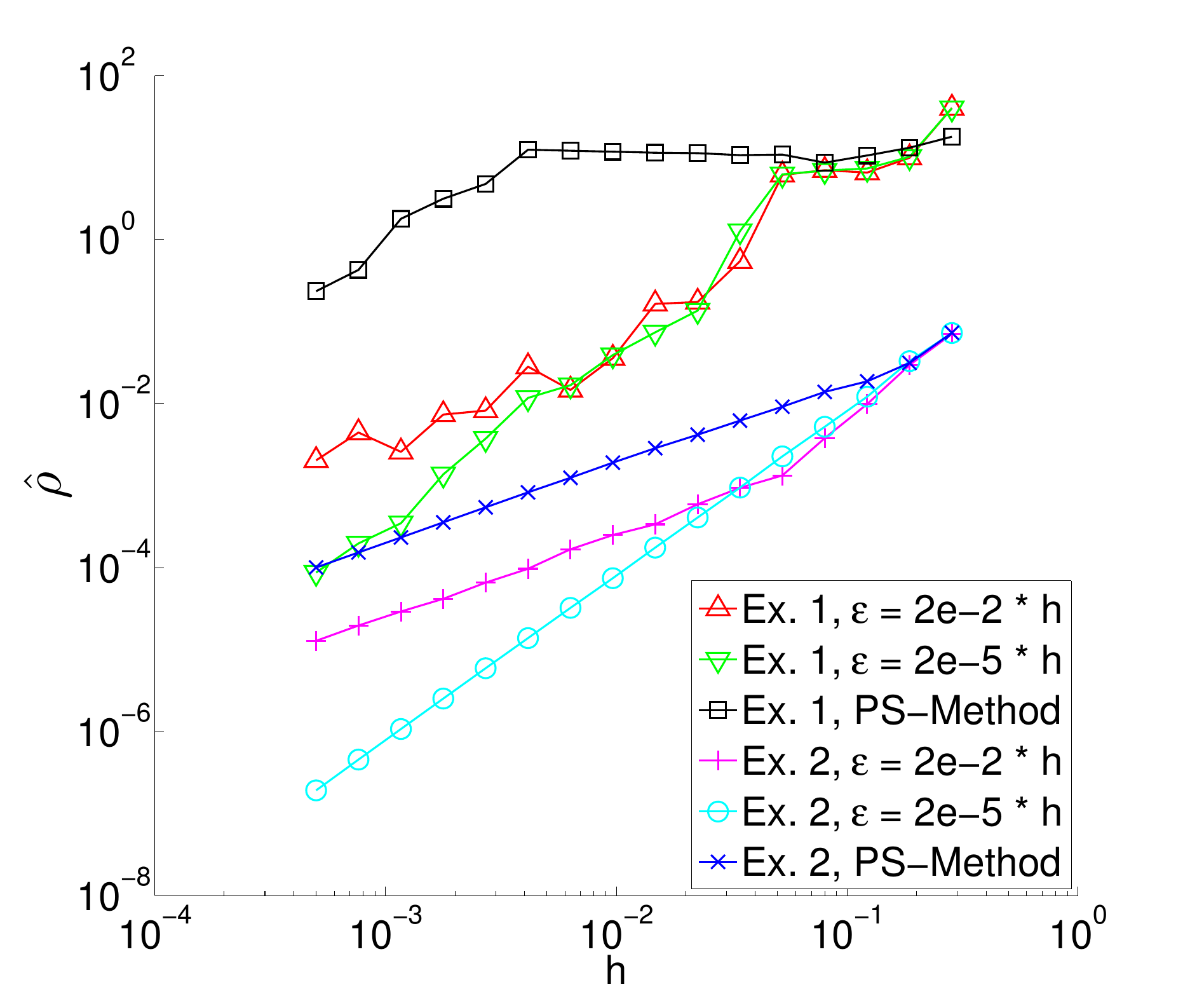}%
    }%
    \hfill%
    \subfloat[Computation times of the algorithm in Fig.~\ref{fig:algo_discrete_approx} vs.\ the PS Method for the examples in Table~\ref{tab:msparams}.\label{fig:flocomp}]%
    {%
      \includegraphics[width=.32\textwidth,keepaspectratio=true]{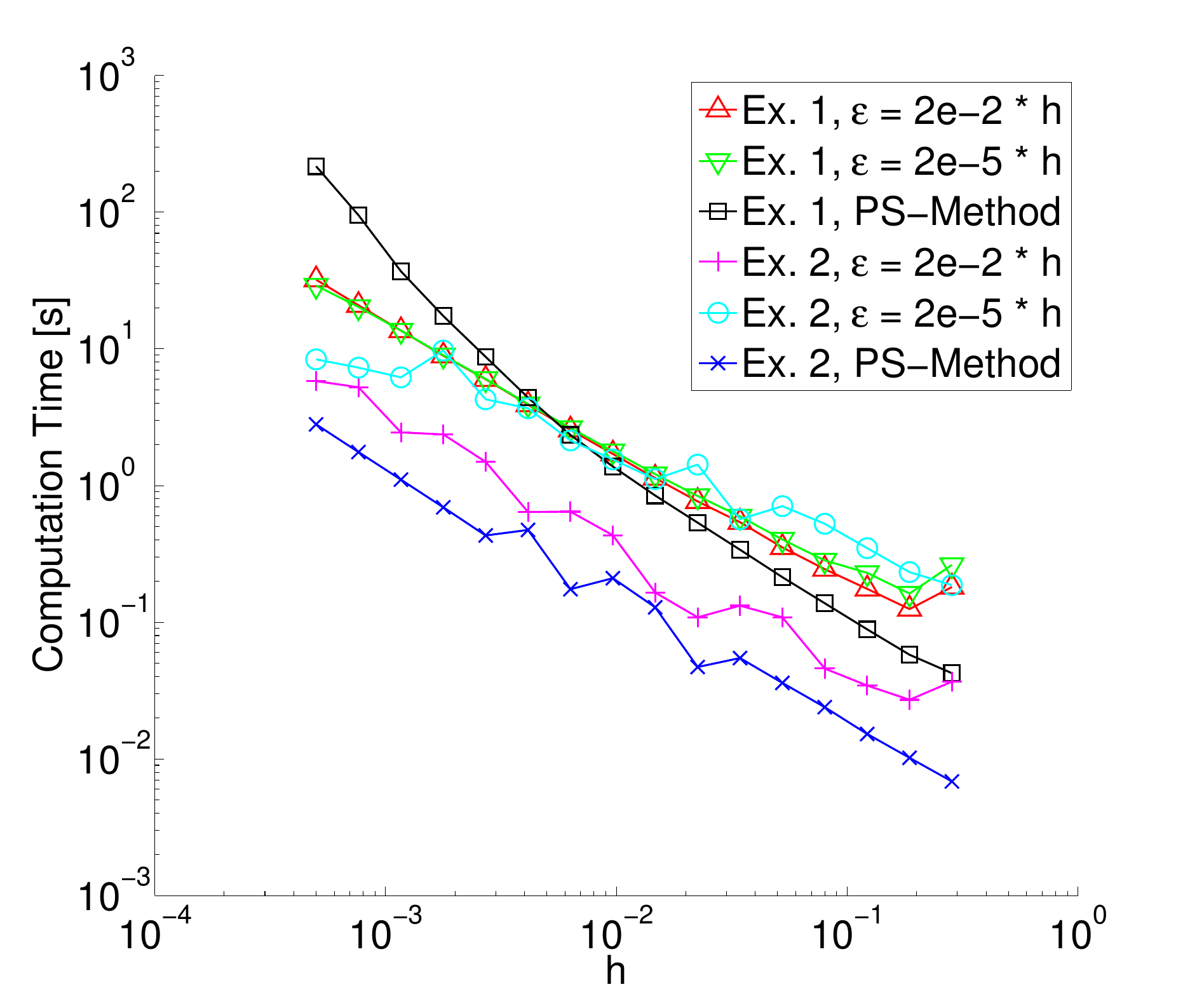}%
    }%
    \caption{A mechanical system (Fig.~\ref{fig:flo}) and a pair of examples (Figs.~\ref{fig:gtex1} and~\ref{fig:gtex2}) chosen to illustrate the accuracy of the algorithm in Fig.~\ref{fig:algo_discrete_approx} vs.\ the PS Method (Figs.~\ref{fig:florho} and~\ref{fig:flomax}) and their computation times (Fig.~\ref{fig:flocomp}).}%
  \end{minipage}%
\end{figure*}

% To illustrate the performance of the algorithm in Fig.~\ref{fig:algo_discrete_approx},
We consider a single degree--of--freedom oscillator consisting of a mass that is externally forced and can impact a plane fixed rigid stop, as in Fig.~\ref{fig:flo}.
The state of the oscillator is the position, $x(t) \in \R$, and velocity, $\dot{x}(t) \in \R$, of the mass.
The oscillator is forced with a control $u \in BV(\R,\R)$.
The oscillator is modeled as a controlled hybrid system with a single mode $D = \set{\parenb{x(t),\dot{x}(t)} \in \R^2 \mid x(t) \leq x_{\text{max}}}$, and single guard corresponding to the mass impacting the stop with non--negative velocity $G = \setb{\parenb{x(t),\dot{x}(t)} \in \R^2 \mid x(t) = x_{\text{max}}, \dot{x}(t) \geq 0}$.
% The oscillator is modeled as a controlled hybrid system with a single mode, denoted~$D$, and single guard corresponding to the mass impacting the stop with non--negative velocity, denoted~$G$:
% \begin{equation}
%   \begin{gathered}
%     D = \set{\parenb{x(t),\dot{x}(t)} \in \R^2 \mid x(t) \leq x_{\text{max}}}\\
%     G = \set{\parenb{x(t),\dot{x}(t)} \in \R^2 \mid x(t) = x_{\text{max}}, \dot{x}(t) \geq 0}
%   \end{gathered}
% \end{equation}
Upon impact, the state is updated using the reset map $R(x,\dot{x}) = \parenb{x,-c\, \dot{x}}$, where $c \in [0,1]$ is the coefficient of restitution.
Within the single domain, the dynamics of the system are governed by $\ddot{x}(t) + 2 a \dot{x}(t) + \omega^2x(t) = m^{-1}\, u(t)$,
where $\omega = \sqrt{m^{-1} k}$, $a = 0.5\, m^{-1}\, \mu$, $k$ is the spring constant, and $\mu$ is the damping coefficient.

% \subsubsection{Analytical Solution}

Given an initial condition $\parenb{x(t_0),\dot{x}(t_0)} = \parenb{x_0,\dot{x}_0} \in D$, the oscillator's motion is analytically determined by:
\begin{equation}
  x(t) = e^{-at}\parenb{A_n \cos(\tilde{\omega} t) + B_n \sin(\tilde{\omega} t)}
  + \tilde{\omega}^{-1}\int_0^t u(s) e^{-a(t - s)}\sin\parenb{\tilde{\omega}(t - s)}\diff{s}
\end{equation}
for each $t \in [t_{n-1},t_n)$, where $\tilde{\omega} = \sqrt{\omega^2 - a^2}$ (assuming that the damping is sub--critical), with $t_n$ such that $x(t_n^-) = x_{\text{max}}$ for each $n \in \N$, and $A_n$ and $B_n$ are determined by the given initial conditions when $n = 0$, or those determined by applying the reset map to $x(t_n^-)$ when $n \geq 1$.
Note that determining the impact times can be done analytically.
The analytical solution holds provided that the mass does not stick to the stop, since in that case the dynamics are given by $\ddot{x}(t) + 2a\dot{x}(t) + \omega^2 x(t) = m^{-1}\, \parenb{u(t) + \lambda(t)}$, where $\lambda(t) \in \R$ denotes the force generated by the stop to prevent movement.
This equation holds as long as $x(t) = x_{\text{max}}$, $\dot{x}(t) = \ddot{x}(t) = 0$, and the reaction of the stop is negative, i.e.\ $\lambda(t) \geq m\, \omega^2\, x_{\text{max}}$.
For the contact to cease, $\lambda(t) - m\, \omega^2\, x_{\text{max}}$ must become zero and change sign.
Once this happens, the analytical solution can be used again to construct the motion of the mass with the initial condition $(x_{\text{max}},0)$.
%% sam-1.20, 3.1
% Note that the conditions for this analytical solution to be defined also ensure the hypotheses of Proposition~\ref{prop:orb} are satisfied, and hence the analytically--determined execution is orbitally stable.
%% HG orbital stability
% Note that Theorem~III.2 in~\cite{Lygeros2003} can be used to establish the orbital stability near the analytical execution \humberto{because?} since it passes transversally through the guard \sam{Ram should verify / confirm the mass doesn't graze the mechanical stop}.

% \subsubsection{The Paoli--Schatzman (PS) Method}

Assuming that the forcing $u$ is continuous (an assumption that is violated by many control schemes such as ones generated via optimal control) a convergent numerical simulation scheme, which we call the PS Method, to determine the position of a mechanical system with unilateral constraints was proposed in~\cite{PaoliSchatzman2003ii}.
Given a step--size $\stepsymbol > 0$ and $t_k = t_0 + h\,k$ for each $k \in \N$, their approach is a two--step method that computes a set of positions, $z_{\text{PS}}\colon \set{t_k}_{k \in \N} \to \R$, by $z_{\text{PS}}(t_0) = x_0$ and:
\begin{equation}
  \begin{aligned}
    z_{\text{PS}}(t_1) &= x_0 + \dot{x}_0 \stepsymbol + \frac{h^2}{2}\parenb{u(0) - 2 a \dot{x}_0 - \omega^2 x_0},\\
    z_{\text{PS}}(t_{k+1}) &= - c\, z_{\text{PS}}(t_{k-1}) + \min\setb{y_{\text{PS}}(t_{k}), (1+c)x_{\text{max}}},\\
    y_{\text{PS}}(t_{k}) &= \frac{1}{1 + a \stepsymbol} \Bigl(h^2 u(t_{k}) + (2 - \stepsymbol^2\omega^2) z_{\text{PS}}(t_{k}) - \parenb{(1-c) - (1+c)a\,h} z_{\text{PS}}(t_{k-1})\Bigr).
  \end{aligned}
\end{equation}

\begin{table}[t]
  \centering
  \caption{Parameters used for the simulations of the forced linear oscillator with stop.}
  \label{tab:msparams}
  \begin{tabular}{c!{\vrule width 1.5pt}c|c|c|c|c|c|c|c}
    & $a$ & $c$ & $t_{\text{max}}$ & $u(t)$ & $x_0$ & $\dot{x}_0$ & $x_{\text{max}}$ & $\omega$ \\
    \noalign{\hrule height 1.5pt}
    Example $1$ & $0.05$ & $0.9$ & $40\pi$ & $20 \cos(\frac{5}{2}t)$ & $11.36$ & $31.4$ & $14$ & $2.5$ \\
    \hline
    Example $2$ & $0.95$ & $0.5$ & $4\pi$ & $ \cos(t)$ & $-0.8$ & $0$ & $-0.8$ & $1$ \\
  \end{tabular}
\end{table}

We illustrate the performance of our approach by considering the two examples described in Table~\ref{tab:msparams} whose solutions, which are defined for all $t \in [0,t_{\text{max}}]$, can be computed analytically.
The position component of the analytical trajectory of each example is plotted in Figs.~\ref{fig:gtex1} and~\ref{fig:gtex2}.
The evaluation of the performance of our algorithm as described in Fig.~\ref{fig:algo_discrete_approx} using $\rx{\rho}$, as in Definition~\ref{def:curve_metric}, is shown in Fig.~\ref{fig:florho}.
To make our approach comparable to the PS Method, for $\dalgo{{\cal A}}$ we use a Runge--Kutta of order two which is called the midpoint method.
We cannot use $\rx{\rho}$ to compare our discrete approximation algorithm to the PS method since the PS method does not compute the velocities of the hybrid system.
Hence, we use the evaluation metric proposed in~\cite{Janin2001} which compares a numerically simulated position trajectory, $z_{\text{pos}}\colon \set{t_k}_{k \in \N} \to \R$, to the analytically computed position trajectory, $x_{\text{analytic}}\colon [0,t_{\text{max}}] \to \R$, at the sample points $\set{t_k}_{k \in \N} \cap [0,t_{\text{max}}]$ as follows:
\begin{equation}
  \label{eq:rhohat}
  \hat{\rho}(z_{\text{pos}},x_{\text{analytic}})
  = \max\setb{\abs{z_{\text{pos}}(t_k) - x_{\text{analytic}}(t_k)} \bigm|
    \set{t_k}_{k \in \N} \cap [0,t_{\text{max}}]}.
\end{equation}
The result of this comparison is illustrated in Fig.~\ref{fig:flomax}.
Finally, the computation time on a 32~GB, 3.1~GHz \emph{Xeon} processor computer for each of the examples as a function of the step--size and relaxation parameter is shown in Fig.~\ref{fig:flocomp}.
Notice in particular that we are able to achieve higher accuracy with respect to the $\hat{\rho}$ evaluation metric at much faster speeds.
In Example~1, for step--sizes $\stepsymbol \leq 10^{-1}$, our numerical simulation method is consistently more accurate by several orders of magnitude and generally several orders of magnitude faster than the PS method.
In Example~2, using a step--size of approximately $h = 10^{-2}$ and relaxation parameter $\rxsymbol = 2 \cdot 10^{-7}$, our numerical simulation achieves a $\hat{\rho}$ value of approximately $10^{-4}$ while taking approximately $0.1$~seconds, whereas the PS method requires a step--size of $\stepsymbol = 5 \cdot 10^{-4}$ which takes approximately $5$~seconds in order to achieve the same level of accuracy.

\subsection{Simultaneous Transitions in Models of Legged Locomotion}
\label{sec:loc}

As a terrestrial agent traverses an environment, its appendages intermittently contact the terrain.
Since the equations governing the agent's motion change with each limb contact, the dynamics are naturally modeled by a controlled hybrid system with discrete modes corresponding to distinct contact configurations.
Because the dynamics of dexterous manipulation are equivalent to that of legged locomotion~\cite{JohnsonKoditschek2013}, such controlled hybrid systems model a broad and important class of dynamic interactions between an agent and its environment.

Legged animals commonly utilize gaits that, on average, involve the simultaneous transition of multiple limbs from aerial motion to ground contact~\cite{Alexander1984,GolubitskyStewart1999}.
Similarly, many multi--legged robots enforce simultaneous leg touchdown via virtual constraints implemented algorithmically~\cite{RaibertChepponis1986,SaranliBuehler2001} or physical constraints implemented kinematically~\cite{KimClark2006,HooverBurden2010}.
% The trajectory of such a gait in the hybrid dynamical system modeling the locomotor's motion passes through the intersection of multiple transition surfaces.
Trajectories modeling such gaits pass through the intersection of multiple transition surfaces in the corresponding controlled hybrid system models.
Therefore simulation of this frequently--observed behavior requires a numerical integration scheme that can accommodate overlapping guards.
The algorithm in Fig.~\ref{fig:algo_discrete_approx} has this capability, and to the best of our knowledge is the only existing algorithm possessing this property.
% To the best of our knowledge, \emph{the algorithm proposed in this paper is the only numerical integration framework that can accommodate overlapping guards in hybrid systems}, and is therefore the only algorithm generally applicable to simulation of multi-legged locomotion or multi-fingered manipulation.
% To demonstrate this capability, we consider a \emph{pronking} gait in the saggital-plane locomotion model illustrated in Figure~\ref{fig:pronk_sch}.
We demonstrate this advanced capability using a \emph{pronking} gait in a saggital--plane locomotion model. % illustrated in Fig.~\ref{fig:pronk_sch}.

Fig.~\ref{fig:pronk_sch} contains an extension of the ``Passive RHex--runner'' in~\cite{SeipelHolmes2006} that allows pitching motion.
A rigid body with mass $m$ and moment--of--inertia $I$ moves in the saggital plane under the influence of gravity $g$.
Linear leg--springs are attached to the body via a frictionless pin joint located symmetrically at distance $d/2$ from the center--of--mass.
The leg--springs are massless with linear stiffness $k$, rest length $\ell$, and make an angle $\psi$ with respect to the body while in the air.
When a foot touches the ground it attaches via a frictionless pin joint, and it detaches when the leg extends to its rest length.

\begin{figure}[t]
  \centering
  \includegraphics[width=.35\linewidth,keepaspectratio=true]{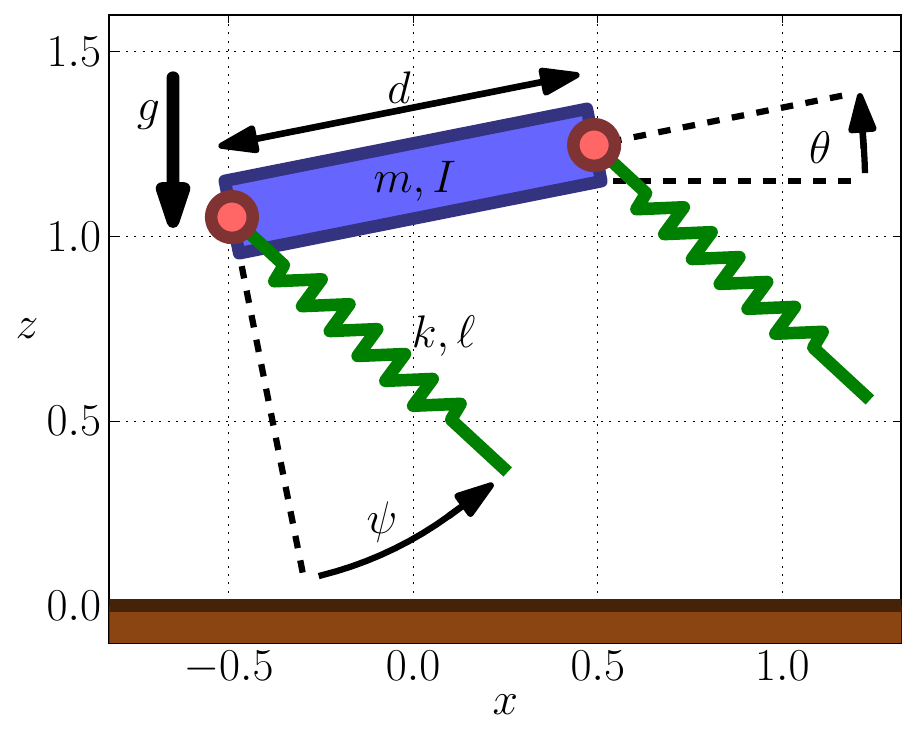}%
  \caption{Schematic for the saggital--plane locomotion model with three mechanical degrees of freedom.}%
  \label{fig:pronk_sch}
\end{figure}
\begin{figure}[t]
  \centering
  \includegraphics[width=.35\linewidth,keepaspectratio=true]{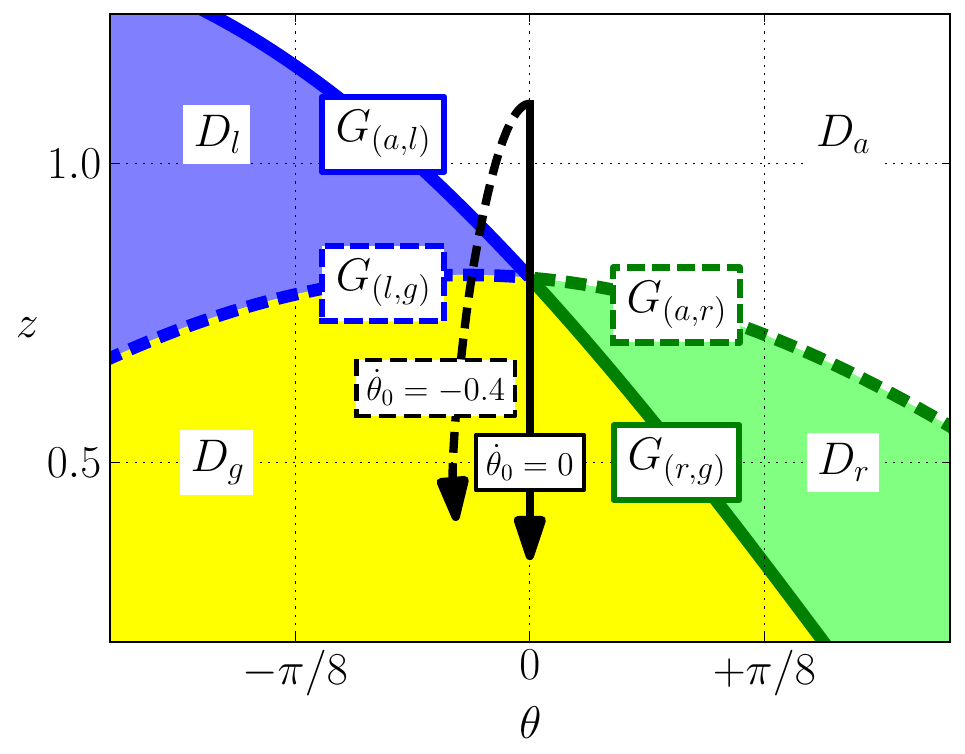}%
  \caption{Projection of guards in $(\theta,z)$ coordinates for transition from aerial domain $D_a$ to ground domain $D_g$ with parameters $d = \ell = 1$, $\psi = \pi/5$.}%
  \label{fig:pronk_grd}
\end{figure}
% \begin{figure*}[tp]
%   \begin{minipage}{\textwidth}
%     \centering
%     \subfloat%
%     [Schematic for the saggital--plane locomotion model with three mechanical degrees of freedom.\label{fig:pronk_sch}]%
%     {\hspace{3em}%
%       \includegraphics[width=.32\linewidth,keepaspectratio=true]{images/pronk_sch}%
%       \hspace{3em}%
%     }
%     \hfill%
%     \subfloat%
%     [Projection of guards in $(\theta,z)$ coordinates for transition from aerial domain $D_a$ to ground domain $D_g$ with parameters $d = \ell = 1$, $\psi = \pi/5$.\label{fig:pronk_grd}]%
%     {\hspace{5em}%
%       \includegraphics[width=.32\linewidth,keepaspectratio=true]{images/pronk_grd}%
%       \hspace{5em}%
%     }
%     \caption{Schematic and discrete mode diagram for the saggital--plane locomotion model.}%
%   \end{minipage}%
% \end{figure*}

% \subsubsection{Pronking Gait}
% \label{sec:loc:pronk}

A \emph{pronk} is a gait wherein all legs touch down and lift off from the ground at the same time~\cite{Alexander1984,GolubitskyStewart1999}.
Due to symmetries in our model, motion with pitch angle $\theta = 0$ for all time is invariant. %detailed in Section~\ref{sec:loc:mdl}.
Therefore periodic orbits for the \emph{spring--loaded inverted pendulum} model in~\cite{GhigliazzaAltendorfer2003} correspond exactly to pronking gaits for our model.
% Therefore the problem of finding a pronking gait reduces to searching for a linear velocity vector $(\dot{x},\dot{z})$ such that initializing the system with both legs on the ground and $\theta = \dot{\theta} = 0, z = \ell\cos(\psi)$ yields the same linear velocity at the next touchdown.
Fig.~\ref{fig:pronk_grd} contains a projection of the guards $G_{(a,l)}$, $G_{(a,r)}$, $G_{(l,g)}$, $G_{(r,g)}$ in $(\theta,z)$ coordinates for the transition from the aerial domain $D_a$ to the ground domain $D_g$ through left stance $D_l$ and right stance $D_r$.
% Though the state space of the model is 10-dimensional, these guards only depend on the height $z$ and pitch angle $\theta$ of the rigid body and are hence faithfully represented in this projection.
The pronking trajectory is illustrated by a downward--pointing vertical arrow, and a nearby trajectory initialized with negative rotational velocity is illustrated by a dashed line.
Fig.~\ref{fig:pronk_seq} contains snapshots from these simulations.

\begin{figure*}[tp]
  \centering
  \mbox{}%
  \hfill%
  \includegraphics[width=.45\linewidth]{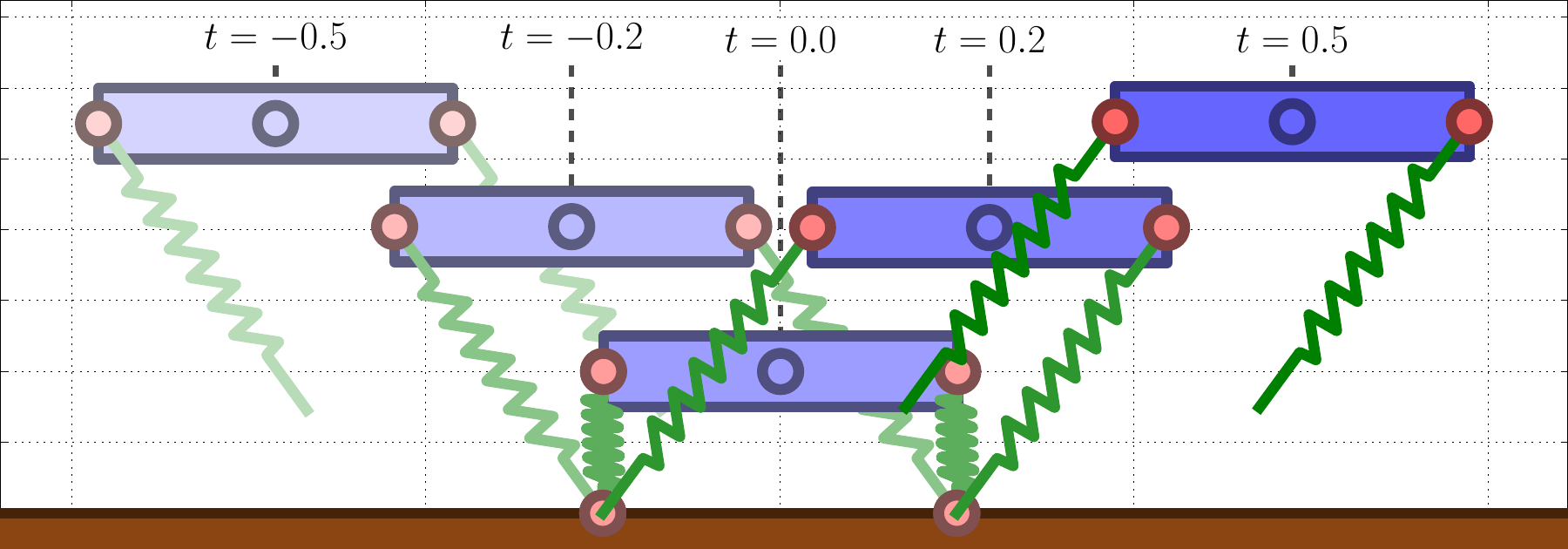}%
  \hfill%
  \includegraphics[width=.45\linewidth]{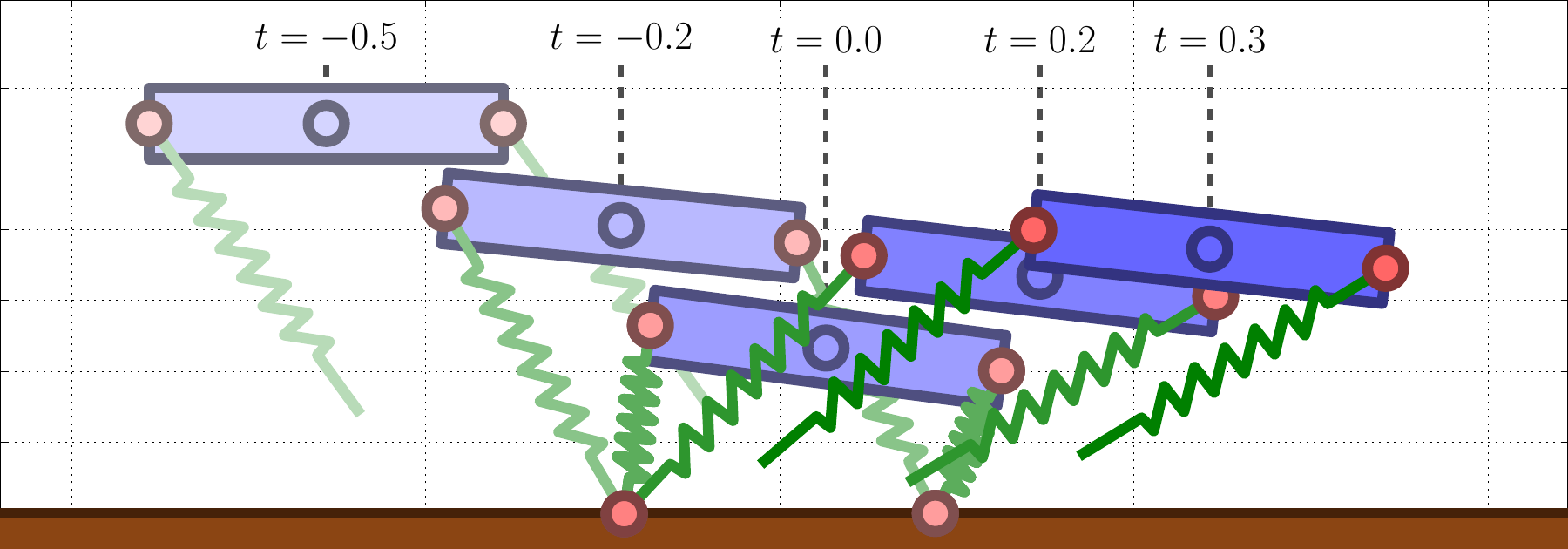}%
  \hfill%
  \mbox{}%
  \caption{%
    Snapshots of \emph{pronk} at discrete transition times
    from initial condition $(x_0,z_0,\theta_0,\dot{x}_0,\dot{z}_0,\dot{\theta}_0) = (0,1.1,0,3.4,0,0)$, parameters $(m,I,k,\ell,d,g,\psi) = (1,1,30,1,1,9.81,\pi/5)$, step size $\stepsymbol = 10^{-3}$, relaxation parameter $\rxsymbol = 10^{-2}$ (left).
    Same as before, but with $\dot{\theta}_0 = -0.4$ (right).
  }
  \label{fig:pronk_seq}
\end{figure*}

The $\dot{\theta}_0 = 0$ trajectory in Fig.~\ref{fig:pronk_grd} clearly demonstrates the need for a simulation algorithm that allows the intersection of multiple transition surfaces.
We emphasize that our state--space metric was necessary to derive a convergent numerical approximation for this execution: since the discrete mode sequence differs for any pair of trajectories arbitrarily close to the $\dot{\theta}_0 = 0$ execution that pass through the interior of $D_l$ and $D_r$, respectively, application of existing trajectory--space metrics~\cite{Tavernini1987, Tavernini2009, SanfeliceTeel2010} would yield a distance larger than unity between the pair.
\label{par:pronk_guard_intersection} Consequently, to the best of our knowledge, no existing provably--convergent numerical simulation algorithm based on a trajectory--space metric is applicable to the $\dot{\theta}_0 = 0$ execution.

Another interesting property of this example is that it is possible to show (by carefully studying the transitions between vector fields through the guards) that the hybrid quotient space ${\cal M}$ is a smooth 6--dimensional manifold near the \emph{pronk} execution, and that the piecewise--defined dynamics yield a continuously--differentiable vector field on this quotient.
% This is a case of a smooth system which is easier to describe and analyze using a controlled hybrid model, rather than a nonlinear model.
%% HG orbital stability
% Since this system is in fact smooth, all orbits are stable (see Theorem~17.9 in~\cite{Lee2003}).
% However, this argument would not apply if the model included damping or impulsive effects at limb touchdown, as the system is not smooth in those cases.

\section{Conclusion}
\label{sec:disc}

We developed an algorithm for the numerical simulation of controlled hybrid systems and proved the uniform convergence of our approximations to executions using a novel metrization of the controlled hybrid system's state space.
The metric and the algorithm impose minimal assumptions on the hybrid system beyond those required to guarantee existence and uniqueness of executions.
As a consequence, our algorithm does not require a specialized mechanism to handle overlapping guards or control inputs: a single code (freely available at \texttt{\small \url{http://purl.org/sburden/hssim}}) will accurately simulate any orbitally stable execution of the hybrid system under investigation.
Beyond their immediate utility, it is our conviction that these tools provide a foundation for formal analysis and computational controller synthesis in hybrid systems.

\bibliographystyle{IEEEtran}
\bibliography{refs}

\end{document}